\documentclass[a4paper,11pt,reqno]{amsart}

\usepackage{amssymb}

\newtheorem{theorem}[equation]{Theorem}

\newtheorem{corollary}[equation]{Corollary}
\newtheorem{proposition}[equation]{Proposition}

\numberwithin{equation}{section}

\theoremstyle{definition}
\newtheorem{definition}[equation]{Definition}
\newtheorem*{example*}{Example}
\newtheorem{example}[equation]{Example}

\newtheorem{remark}[equation]{Remark}
\newtheorem*{remark*}{Remark}

\newcommand{\bN}{{\mathbb N}}
\newcommand{\bZ}{{\mathbb Z}}

\newcommand{\bO}{{\mathbb O}}

\newcommand{\frg}{{\mathfrak g}}

\newcommand{\frgtetra}{{{\mathfrak g}_{\boxtimes}}}

\newcommand{\frt}{{\mathfrak t}}
\newcommand{\frs}{{\mathfrak s}}

\newcommand{\frl}{{\mathfrak l}}

\newcommand{\calS}{{\mathcal S}}

\newcommand{\calA}{{\mathcal A}}

\newcommand{\calL}{{\mathcal L}}
\newcommand{\calK}{{\mathcal K}}
\newcommand{\calV}{{\mathcal V}}
\newcommand{\calB}{{\mathcal B}}

\DeclareMathOperator{\eespan}{span}
\providecommand{\espan}[1]{\eespan\left\{ #1\right\}}

 \newcommand{\tri}{\mathfrak{tri}}
  
   \newcommand{\lrt}{\mathfrak{lrt}}
    \newcommand{\inlrt}{\mathfrak{inlrt}}
 
 \newcommand{\frsl}{{\mathfrak{sl}}}
 
 \newcommand{\frso}{{\mathfrak{so}}}
 
 \newcommand{\frgl}{{\mathfrak{gl}}}

 \newcommand{\frd}{{\mathfrak{d}}}

 \DeclareMathOperator{\ad}{ad}
 \newcommand{\der}{\mathfrak{der}}
 \newcommand{\sder}{\mathfrak{sder}}
 
 \DeclareMathOperator{\frsldos}{{{\mathfrak{sl}_2}}}

 \DeclareMathOperator{\Aut}{Aut}
  \DeclareMathOperator{\Fix}{Fix}
 \DeclareMathOperator{\alg}{alg}

\newenvironment{romanenumerate}
 {\begin{enumerate}
 
 }{\end{enumerate}}

\begin{document}

\title[Lie algebras with $S_3$ or $S_4$-action]{Lie algebras with $S_3$ or $S_4$-action, and generalized Malcev algebras}

\author[Alberto Elduque]{Alberto Elduque$^{\star}$}
 \thanks{$^{\star}$ Supported by the Spanish Ministerio de
 Educaci\'{o}n y Ciencia
 and FEDER (MTM 2007-67884-C04-02) and by the
Diputaci\'on General de Arag\'on (Grupo de Investigaci\'on de
\'Algebra)}
 \address{Departamento de Matem\'aticas e
 Instituto Universitario de Matem\'aticas y Aplicaciones,
 Universidad de Zaragoza, 50009 Zaragoza, Spain}
 \email{elduque@unizar.es}

\author[Susumu Okubo]{Susumu Okubo$^{\ast}$}
 \thanks{$^{\ast}$ Supported in part by U.S.~Department of Energy Grant No.
 DE-FG02-91ER40685.}
 \address{Department of Physics and Astronomy, University of
 Rochester, Rochester, NY 14627, USA}
 \email{okubo@pas.rochester.edu}

\date{January 15, 2008}

\begin{abstract}
Lie algebras endowed with an action by automorphisms of any of the symmetric groups $S_3$ or $S_4$ are considered, and their decomposition into a direct sum of irreducible modules for the given action is studied.

In case of $S_3$-symmetry, the Lie algebras are coordinatized by some nonassociative systems, which are termed \emph{generalized Malcev algebras}, as they extend the classical Malcev algebras. These systems are endowed with a binary and a ternary products, and include both the Malcev algebras and the Jordan triple systems.
\end{abstract}

\maketitle


\section*{Introduction}

The Tetrahedron algebra is an infinite dimensional Lie algebra which
is endowed with a natural action by automorphisms of $S_4$, the
symmetric group of degree $4$ (see \cite{Edinburgh}). Lie algebras
with such an action have been investigated by the authors in
\cite{EO}, where it is shown that these Lie algebras are
coordinatized by a class of nonassociative algebras that include, as
a very important case, the structurable algebras introduced by
Allison \cite{Allison}. These algebras are very useful in providing models of simple Lie algebras,  models which reflect a $S_4$-symmetry (see \cite{EO} and the references there in).

In \cite[Problem 20.10]{ItoPaul} the authors posed the question of how the
Tetrahedron algebra decomposes into a direct sum of irreducible
modules for $S_4$. The first aim of this paper is to study the
decomposition into direct sums of irreducible $S_4$-modules of any
Lie algebra endowed with an action of $S_4$ by automorphisms. This
will be done in Section \ref{se:S4}. The answer for the Tetrahedron
algebra is particularly simple, as only the two irreducible three
dimensional modules for $S_4$ live inside the Lie algebra. This will
be checked in Section \ref{se:Tet}.

On the other hand, given any Lie algebra endowed with an action of
$S_4$ by automorphisms, Klein's $4$-group induces a grading of the
Lie algebra over $\bZ_2\times\bZ_2$. The $(\bar 0,\bar 0)$-component
of this grading, that is, the subspace of elements fixed by any
element in Klein's $4$-group, is naturally endowed with an action of
the symmetric group $S_3$ by automorphisms. Lie algebras with such
an action are studied in Section \ref{se:S3}, where the \emph{generalized Malcev algebras} are introduced.

As already noticed by Mikheev \cite{Mikheev} and Grishkov
\cite{Grishkov}, Malcev algebras appear as coordinate algebras of
Lie algebras with an action of $S_3$ by automorphisms satisfying an
extra condition, which in our terms translates into the fact that, as a module for $S_3$, there is no submodule isomorphic to the alternating one
dimensional module for $S_3$. This will be the subject of Section
\ref{se:S3Malcev}, where the examples of Malcev algebras and Jordan triple systems are discussed. Different examples of Lie algebras with an action
of $S_3$ by automorphisms and of their coordinate algebras will be
given in the last Section \ref{se:S3examples}.

\smallskip

Throughout the paper, all the vector spaces considered will be
defined over a ground field $k$ of characteristic $\ne 2,3$.
Unadorned tensor products will be defined over $k$.

\smallskip

Recall that the symmetric group $S_4$  is the semidirect product of
Klein's $4$-group $V_4=\langle \tau_1,\tau_2\rangle$ and the
symmetric group $S_3=\langle \varphi,\tau\rangle$. Here
\[
\begin{split}
\tau_1&=(12)(34),\\
\tau_2&=(23)(14),\\
\varphi&=(123): 1\mapsto 2\mapsto 3\mapsto 1\\
\tau&=(12).
\end{split}
\]
The symmetric group $S_3$ appears both as a subgroup and as a
quotient $S_4/V_4$ of the symmetric group $S_4$.

There are exactly five non-isomorphic irreducible modules for the
symmetric group $S_4$ (see, for instance, \cite{FH}):
\begin{itemize}
\item $U=ku$, the \emph{trivial module}: $\sigma(u)=u$ for any $\sigma\in
S_4$.
\item $U'=ku'$, the \emph{alternating module}: $\sigma(u')=(-1)^\sigma u'$
for any $\sigma\in S_4$, where $(-1)^\sigma$ denotes the signature
of the permutation $\sigma$.
\item $W=\{(\alpha_1,\alpha_2,\alpha_3)\in k^3: \alpha_1+\alpha_2+\alpha_3=0\}$,
the two dimensional irreducible module. This is a natural module for
$S_3$
($\sigma\bigl((\alpha_1,\alpha_2,\alpha_3)\bigr)=(\alpha_{\sigma^{-1}(1)},
\alpha_{\sigma^{-1}(2)},\alpha_{\sigma^{-1}(3)})$ and hence a module
for $S_4$ in which the elements of $V_4$ act trivially.
\item $V=\{(\alpha_1,\alpha_2,\alpha_3,\alpha_4)\in k^4:
\alpha_1+\alpha_2+\alpha_3+\alpha_4=0\}$, the \emph{standard
module}.
\item $V'=U'\otimes V$.
\end{itemize}

The modules $U$, $U'$ and $W$ form a family of representatives of
the isomorphism classes of the irreducible modules for $S_3$.

\bigskip
\section{Lie algebras with $S_4$-action}\label{se:S4}

Let $\frg$ be a Lie algebra over our ground field $k$  endowed with
a group homomorphism
\[
S_4\rightarrow \Aut(\frg).
\]

As in \cite{EO}, the action of Kleins's $4$-group $V_4$ gives a
$\bZ_2\times\bZ_2$-grading on $\frg$:
\begin{equation}\label{eq:z22grading}
\frg=\frt\oplus\frg_0\oplus\frg_1\oplus\frg_2,
\end{equation}
where
\begin{equation}\label{eq:tg0g1g2}
\begin{split}
\frt&=\{x\in\frg : \tau_1(x)=x,\, \tau_2(x)=x\}\
       (=\frg_{(\bar 0,\bar 0)}),\\
\frg_0&=\{x\in\frg : \tau_1(x)=x,\, \tau_2(x)=-x\}\
        (=\frg_{(\bar 1,\bar 0)}),\\
\frg_1&=\{x\in\frg : \tau_1(x)=-x,\, \tau_2(x)=x\}\
        (=\frg_{(\bar 0,\bar 1)}),\\
\frg_2&=\{x\in\frg : \tau_1(x)=-x,\, \tau_2(x)=-x\}\
        (=\frg_{(\bar 1,\bar 1)}).
\end{split}
\end{equation}
(Here, the subindices $0,1,2$ must be considered modulo $3$.)

Assume that the action of $V_4$ is not trivial, as otherwise the
$S_4$-action is just an action of $S_3\simeq S_4/V_4$, and let then
$A$ denote the subspace $\frg_0$. For $x\in \frg_0$, define
\begin{equation}\label{eq:iota123}
\iota_0(x)=x\in \frg_0,\quad \iota_1(x)=\varphi(\iota_0(x))\in
\frg_1,\quad \iota_2(x)=\varphi^2(\iota_0(x))\in\frg_2.
\end{equation}
Thus
\[
\frg=\frt\oplus \bigl(\oplus_{i=0}^2\iota_i(A)\bigr).
\]

As shown in \cite{EO}, $A$ becomes an algebra with involution
$(A,\cdot,-)$ where:
\begin{itemize}
\item The involution (involutive antiautomorphism) is given by
\[
\iota_0(\bar x)=-\tau(\iota_0(x))
\]
for any $x\in A$. Note that $\tau(\frg_0)= \frg_0$ and that this
gives immediately, since $\varphi\tau=\tau\varphi^2$, the following
actions:
\begin{equation}\label{eq:bar}
\tau(\iota_0(x))=-\iota_0(\bar x),\quad
\tau(\iota_1(x))=-\iota_2(\bar x),\quad
\tau(\iota_2(x))=-\iota_1(\bar x).
\end{equation}

\item The multiplication is given by
\begin{equation}\label{eq:multiplication}
\iota_0(\overline{x\cdot y})=[\iota_1(x),\iota_2(y)],
\end{equation}
for any $x,y\in A$.
\end{itemize}

Moreover, consider the space of \emph{Lie related triples} (see \cite{AllisonFaulkner}):
\[
\begin{split}
\lrt(A,\cdot,-)&=\{(d_0,d_1,d_2)\in\frgl(A)^3:\\
&\qquad \bar d_i(x\cdot y)=d_{i+1}(x)\cdot y+x\cdot d_{i+2}(y),\
\forall x,y\in A,\ \forall i=0,1,2\},
\end{split}
\]
where
\[
\bar d(x)=\overline{d(\bar x)}.
\]
This space $\lrt(A,\cdot,-)$
is a Lie algebra under componentwise bracket, and it comes endowed
with the action of $S_3$ by automorphisms given by:
\begin{equation}\label{eq:phitaulrt}
\begin{split}
\varphi\bigl((d_0,d_1,d_2)\bigr)&=(d_2,d_0,d_1),\\
\tau\bigl((d_0,d_1,d_2)\bigr)&=(\bar d_0,\bar d_2,\bar d_1),
\end{split}
\end{equation}
which is compatible with the action of $S_3$ on $\frt$, that is,
\[
\rho\bigl(\sigma(d)\bigr)=\sigma\bigl(\rho(d)\bigr),
\]
for any $\sigma\in S_3$ and $d\in\frt$ (see \cite[eqs. (1.8) and
(2.3)]{EO}), where $\rho$ is the linear map:
\[
\begin{split}
\rho: \frt&\longrightarrow \frgl(A)^3\\
d&\mapsto \bigl(\rho_0(d),\rho_1(d),\rho_2(d)\bigr),
\end{split}
\]
with
\[\iota_i\bigl(\rho_i(d)(x)\bigr)=[d,\iota_i(x)]
\]
for any $d\in
\frt$ and $x\in A$. It turns out that $\rho$ is a Lie algebra
homomorphism and that $\rho(\frt)$ is contained in $\lrt(A,\cdot,-)$.
Moreover, there appears a skew-symmetric bilinear map
\[
\begin{split}
\delta:A\times A&\longrightarrow \lrt(A,\cdot,-)\\
(x,y)\, &\mapsto \rho\bigl([\iota_0(x),\iota_0(y)]\bigr)
     =\bigl(\delta_0(x,y),\delta_1(x,y),\delta_2(x,y)\bigr).
\end{split}
\]

\begin{theorem}\label{th:deltasNLRTA} \textup{(\cite[Theorem
2.4]{EO})}
For any $a,b,x,y,z\in A$ and $i,j=0,1,2$:
\begin{romanenumerate}
\item
$\bigl[\delta_i(a,b),\delta_j(x,y)\bigr]
 =\delta_j\bigl(\delta_{i-j}(a,b)(x),y\bigr)+
 \delta_j(x,\delta_{i-j}(a,b)(y)\bigr)$,
\item $\delta_0(\bar x,y\cdot z)+\delta_{1}(\bar y,z\cdot x)
   +\delta_{2}(\bar z,x\cdot y)=0$,
\item $\delta_0(x,y)(z)+\delta_0(y,z)(x)+\delta_0(z,x)(y)=0$,
\item $\delta_1(x,y)=L_{\bar y}L_x-L_{\bar x}L_y$,
\item $\delta_2(x,y)=R_{\bar y}R_x-R_{\bar x}R_y$,
\item $\overline{\delta_i(x,y)}=\delta_{-i}(\bar x,\bar y)$ (or
$\tau\bigl(\delta(x,y)\bigr)=\delta(\bar x,\bar y)$),
\end{romanenumerate}
where $L_x$ and $R_x$ denote, respectively, the left and right
multiplication by $x$ in the algebra $(A,\cdot)$.
\end{theorem}

The $4$-tuple $(A,\cdot,-,\delta)$ is then called a \emph{normal Lie
related triple algebra} (see \cite{Oku05}), or normal LRTA for
short.

\smallskip

Conversely, given a normal LRTA $(A,\cdot,-,\delta)$, consider three
copies of $A$: $\iota_i(A)$, $i=0,1,2$, and the Lie subalgebra of
\emph{inner Lie related triples}:
\[
\inlrt(A,\cdot,-,\delta)=\sum_{i=0}^2\varphi^i\bigl(\delta(A,A)\bigr).
\]
Then the vector space
\[
\frg(A,\cdot,-,\delta)
=\inlrt(A,\cdot,-,\delta)\oplus\bigl(\oplus_{i=0}^2\iota_i(A)\bigr)
\]
is a Lie algebra \cite[Theorem 2.6]{EO} with bracket determined by:
\begin{equation}\label{eq:bracketgAdelta}
\begin{split}
&\bullet\ \text{$\inlrt(A,\cdot,-,\delta)$ is a Lie subalgebra},\\
&\bullet\ [(d_0,d_1,d_2),\iota_i(x)]=\iota_i\bigl(d_i(x)\bigr),\
\forall (d_0,d_1,d_2)\in\inlrt(A,\cdot,-,\delta),\ \forall x\in A,\\
&\bullet\ [\iota_i(x),\iota_{i+1}(y)]=\iota_{i+2}(\overline{x\cdot
y})\ \forall x,y\in A,\ \forall i=0,1,2,\\
&\bullet\ [\iota_i(x),\iota_i(y)]=\varphi^i\bigl(\delta(x,y)\bigr),
\forall x,y\in A,\ \forall i=0,1,2.
\end{split}
\end{equation}
Moreover, the symmetric group $S_4$ acts naturally by automorphisms
on $\frg(A,\cdot,-,\delta)$ by means of:
\begin{equation}\label{eq:S4gAdelta}
\begin{split}
&\bullet\ \text{$V_4$ acts trivially on $\inlrt(A,\cdot,-,\delta)$
and $\varphi$ and $\tau$ act by \eqref{eq:phitaulrt},}\\
&\bullet\ \varphi\bigl(\iota_i(x)\bigr)=\iota_{i+1}(x),\ \forall i=0,1,2,\\
&\bullet\ \tau(\iota_0(x))=-\iota_0(\bar x),\
\tau(\iota_1(x))=-\iota_2(\bar x),\
 \tau(\iota_2(x))=-\iota_1(\bar x),\\
&\bullet\ \tau_1\bigl(\iota_0(x)\bigr)=\iota_0(x),\
  \tau_1\bigl(\iota_i(x)\bigr)=-\iota_i(x)\ \text{for $i=1,2$},\\
&\bullet\ \tau_2\bigl(\iota_1(x)\bigr)=\iota_1(x),\
  \tau_2\bigl(\iota_i(x)\bigr)=-\iota_i(x)\ \text{for $i=0,2$},
\end{split}
\end{equation}
for any $x\in A$.

One can also consider the larger Lie algebra
\begin{equation}\label{eq:tildefrg}
\tilde\frg(A,\cdot,-,\delta)
=\lrt(A,\cdot,-,\delta)\oplus\bigl(\oplus_{i=0}^2\iota_i(A)\bigr)
\end{equation}
with the same bracket given in \eqref{eq:bracketgAdelta}, which
again is endowed with an action of $S_4$ by automorphisms given by
\eqref{eq:S4gAdelta}. Besides, if $\frg$ is a Lie algebra with an
action of $S_4$ and $(A,\cdot,-,\delta)$ is the associated normal
LRTA, then the Lie algebra homomorphism $\rho:\frt\rightarrow
\lrt(A,\cdot,-)$ extends to a Lie algebra homomorphism
\[
\tilde\rho:\frg\longrightarrow \tilde\frg(A,\cdot,-,\delta)
\]
compatible with the action of $S_4$, and such that
$\tilde\rho\bigl(\iota_i(x)\bigr)=\iota_i(x)$ for any $x\in A$ and
$i=0,1,2$.

\begin{proposition}
Let $\frg$ be a Lie algebra with $S_4$-action and let $\frg=\frt\oplus \frg_0\oplus\frg_1\oplus\frg_2$ be the associated $\bZ_2\times\bZ_2$-grading as in \eqref{eq:z22grading}. Assume that $\frg_0\ne 0$ and
that $\frg$ has no proper ideals invariant under the action of
$S_4$ (this happens, in particular, if $\frg$ is simple). Then the
homomorphism $\tilde\rho$ is one-to-one and its image is
$\frg(A,\cdot,-,\delta)$.
\end{proposition}
\begin{proof}
Since $\tilde\rho$ is a homomorphism of Lie algebras with
$S_4$-action, its kernel $\ker\tilde\rho$ is an ideal invariant
under the action of $S_4$, so it is trivial. Hence $\tilde\rho$ is
one-to-one. On the other hand, the subalgebra of $\frg$ generated by
$\frg_0\oplus\frg_1\oplus\frg_2$ is an ideal invariant under $S_4$,
so it is the whole $\frg$, and hence $\tilde\rho(\frg)$ is the
subalgebra generated by $\oplus_{i=0}^2\iota_i(A)$, which is
precisely $\frg(A,\cdot,-,\delta)$.
\end{proof}

\smallskip

Now we are ready to decompose the subspace $\frg_0\oplus\frg_1\oplus \frg_2$ in \eqref{eq:z22grading} as a
direct sum of irreducible modules for the action of $S_4$.

\begin{theorem}\label{th:g0g1g2}
Let $\frg$ be a Lie algebra endowed with an action of $S_4$ by
automorphisms, and let $0\ne x\in A=\frg_0$. Then:
\begin{romanenumerate}
\item If $\bar x=-x$, then $\sum_{i=0}^2k\iota_i(x)$ is a
$S_4$-module isomorphic to the standard module $V$.
\item If $\bar x=x$, then $\sum_{i=0}^2k\iota_i(x)$ is a
$S_4$-module isomorphic to $V'=U'\otimes V$.
\end{romanenumerate}
\end{theorem}
\begin{proof}
Equations \eqref{eq:tg0g1g2}, \eqref{eq:iota123} and \eqref{eq:bar}
show that in both cases the subspace $\sum_{i=0}^2k\iota_i(x)$ is invariant under
the action of $S_4$. In the first case the assignment
\[
\iota_0(x)\mapsto (1,1,-1,-1),\ \iota_1(x)\mapsto (-1,1,1,-1),\
\iota_2(x)\mapsto (1,-1,1,-1),
\]
 provides the required isomorphism,
while in the second case, the assignment $\iota_0(x)\mapsto
u'\otimes (1,1,-1,-1),\ \ldots$, gives the result.
\end{proof}

\begin{corollary}\label{co:g0g1g2}
Let $\frg$ be a Lie algebra endowed with an action of $S_4$ by
automorphisms. Then
$\oplus_{i=0}^2\iota_i(A)=\frg_0\oplus\frg_1\oplus\frg_2$ is a
direct sum of copies of the two irreducible three-dimensional
modules for $S_4$.
\end{corollary}

\smallskip

Let us turn now our attention to the action of $S_4$ (actually of
$S_3\simeq S_4/V_4$) on the Lie algebra $\lrt(A,\cdot,-,\delta)$.

On the two one-dimensional irreducible modules (the trivial one $U$
and the alternating one $U'$ for $S_4$ (or $S_3$)), the cycle
$\varphi$ acts trivially, while $\varphi$ acts with minimal
polynomial $X^2+X+1$ on the two-dimensional irreducible module. But
for $(A,\cdot,-)$ an algebra with involution:
\[
\begin{split}
\{ (d_0,d_1,d_2)&\in\lrt(A,\cdot,-):
\varphi\bigl((d_0,d_1,d_2)\bigr)=(d_0,d_1,d_2)\}\\
&=\{(d,d,d)\in\frgl(A)^3: \bar d(x\cdot y)=d(x)\cdot y+x\cdot d(y)\
\forall x,y\in A\}\\
&=\{(d,d,d)\in\frgl(A)^3:  d(x* y)=d(x)* y+x* d(y)\ \forall x,y\in
A\}
\end{split}
\]
where $x*y=\overline{x\cdot y}$ for any $x,y\in A$. Hence, the
subspace of fixed elements by $\varphi$ in $\lrt(A,\cdot,-)$ is
naturally isomorphic to the Lie algebra of derivations of the
algebra $(A,*)$:
\[
\Fix_{\lrt(A,\cdot,-)}(\varphi)\simeq \der(A,*).
\]

Given an algebra with involution $(A,\cdot,-)$, $\der(A,\cdot,-)$
will denote the Lie algebra:
\[
\der(A,\cdot,-)=\{d\in \der(A,\cdot): \bar d=d\}.
\]

\begin{proposition}\label{pr:S3lrtA}
Let $(A,\cdot,-)$ be an algebra with involution. Then:
\begin{romanenumerate}
\item The direct sum of the trivial $S_4$-submodules of $\lrt(A,\cdot,-)$
is the subspace $ \{(d,d,d): d\in \der(A,\cdot,-)\}$.

\item The direct sum of the $S_4$-submodules of $\lrt(A,\cdot,-)$
isomorphic to the alternating module is $ \{(d,d,d): d\in
\sder(A,\cdot,-)\}$, where $\sder(A,\cdot,-)$ is the subspace of
skew-derivations of $(A,\cdot,-)$:
\[
\sder(A,\cdot,-)=\{ d\in \frgl(A): d(x\cdot y)=-d(x)\cdot y- x\cdot
d(y)\ \text{and}\ \bar d=-d\}.
\]

\item The direct sum of the irreducible $S_4$-submodules of
$\lrt(A,\cdot,-)$ isomorphic to the two-dimensional irreducible
module $W$ is
\[
\{(d_0,d_1,d_2)\in \lrt(A,\cdot,-): d_0+d_1+d_2=0\}.
\]
\end{romanenumerate}
\end{proposition}
\begin{proof}
The direct sum of the trivial submodules is
\[
\begin{split}
\{(d_0,d_1,d_2)\in \lrt(A,\cdot,-)&:
\varphi\bigl((d_0,d_1,d_2)\bigr)=(d_0,d_1,d_2)=
\tau\bigl((d_0,d_1,d_2)\bigr)\}\\
&=\{(d_0,d_1,d_2)\in\lrt(A,\cdot,-): d_0=d_1=d_2=\bar d_0\}
\end{split}
\]
because of \eqref{eq:phitaulrt}, and this gives (i).

Also, the direct sum of the alternating submodules is
\[
\begin{split}
\{(d_0,d_1,d_2)\in \lrt(A,\cdot,-)&:
\varphi\bigl((d_0,d_1,d_2)\bigr)=(d_0,d_1,d_2)=
-\tau\bigl((d_0,d_1,d_2)\bigr)\}\\
&=\{(d_0,d_1,d_2)\in\lrt(A,\cdot,-): d_0=d_1=d_2=-\bar d_0\},
\end{split}
\]
and this proves (ii).

Finally, the direct sum of the $S_4$-submodules of $\lrt(A,\cdot,-)$
which are isomorphic to the two-dimensional irreducible module $W$
is the kernel of $1+\varphi+\varphi^2$, and this gives the result in
(iii).
\end{proof}

Let us pause to give an example of a normal LRTA
$(A,\cdot,-,\delta)$ where $\der(A,\cdot,-)$ is strictly contained
in the Lie algebra of derivations $\der(A,\cdot)$.

\begin{example}
Let $\frg$ be any $\bZ_2\times \bZ_2$-graded Lie algebra:
\[
\frg=\frg_{(\bar 0,\bar 0)}\oplus \frg_{(\bar 1,\bar 0)}
\oplus\frg_{(\bar 0,\bar 1)}\oplus\frg_{(\bar 1,\bar 1)},
\]
with $\frg_{(\bar 0,\bar 1)}\ne 0\ne \frg_{(\bar 1,\bar 1)}$, and
let $\nu$ and $\mu$ the order two automorphisms of $\frg$ given by:
\[
\begin{split}
\nu&=\begin{cases}id&\text{on $\frg_{(\bar 0,\bar
        0)}\oplus\frg_{(\bar 1,\bar 0)}$,}\\
        -id&\text{on $\frg_{(\bar 0,\bar 1)}\oplus\frg_{(\bar 1,\bar
        1)}$,}\end{cases}\\
\mu&=\begin{cases}id&\text{on $\frg_{(\bar 0,\bar
        0)}\oplus\frg_{(\bar 0,\bar 1)}$,}\\
        -id&\text{on $\frg_{(\bar 1,\bar 0)}\oplus\frg_{(\bar 1,\bar
        1)}$.}\end{cases}
\end{split}
\]
The symmetric group $S_4$ acts by automorphisms on $\frg^3$ as
follows:
\[
\begin{split}
\tau_1\bigl((x,y,z)\bigr)&=(x,\nu(y),\nu(z)),\\
\tau_2\bigl((x,y,z)\bigr)&=(\nu(x),y,\nu(z)),\\
\varphi\bigl((x,y,z)\bigr)&=(z,x,y),\\
\tau\bigl((x,y,z)\bigr)&=(\mu(x),\mu(z),\mu(y)).
\end{split}
\]
(Compare to \cite[Example 3.5]{EO}.) Then the decomposition in \eqref{eq:z22grading} is given by:
\[
\begin{split}
\frt&= \bigl(\frg_{(\bar 0,\bar 0)}
   \oplus\frg_{(\bar 1,\bar 0)}\bigr)^3,\\
\frg_0&=\{(x,0,0): x\in \frg_{(\bar 0,\bar 1)}
   \oplus\frg_{(\bar 1,\bar 1)}\},\\
\frg_1&=\{(0,x,0): x\in \frg_{(\bar 0,\bar 1)}
   \oplus\frg_{(\bar 1,\bar 1)}\},\\
\frg_2&=\{(0,0,x): x\in \frg_{(\bar 0,\bar 1)}
   \oplus\frg_{(\bar 1,\bar 1)}\},
\end{split}
\]
so that $A=\frg_{(\bar 0,\bar 1)} \oplus\frg_{(\bar 1,\bar 1)}$ is a
normal LRTA with
\[
x\cdot y=0,\qquad \bar x=-\mu(x),
\]
for any $x\in A$.

The Lie algebra of derivations of $(A,\cdot)$ is then the whole
general linear algebra:
\[
\der(A,\cdot)=\frgl(A),
\]
while
\[
\begin{split}
\der(A,\cdot,-)&=\{f\in\frgl(A): f(\bar x)=\overline{f(x)}\ \forall
x\}\\
 &=\{f\in\frgl(A): f\mu=\mu f\}\\
 &=\{f\in\frgl(A): f\bigl(\frg_{(\bar 0,\bar 1)}\bigr)\subseteq \frg_{(\bar 0,\bar
 1)},\ f\bigl(\frg_{(\bar 1,\bar 1)}\bigr)\subseteq \frg_{(\bar 1,\bar
 1)}\},
\end{split}
\]
and
\[
\begin{split}
\sder(A,\cdot,-)&=\{f\in\frgl(A): f\mu=-\mu f\}\\
 &=\{f\in\frgl(A): f\bigl(\frg_{(\bar 0,\bar 1)}\bigr)\subseteq \frg_{(\bar 1,\bar
 1)},\ f\bigl(\frg_{(\bar 1,\bar 1)}\bigr)\subseteq \frg_{(\bar 0,\bar
 1)}\}.\quad \qed
\end{split}
\]
\end{example}

\medskip

In \cite[Theorem 2.4]{EO} it is proved that the unital normal LRTA's
are precisely the structurable algebras. In this case, the previous
Proposition becomes simpler.

\begin{corollary}\label{co:S3lrta}
Let $(A,\cdot,-)$ be a unital algebra with involution, then \newline
$\sder(A,\cdot,-)=0$.

In particular, the alternating module does not appear in the
decomposition of $\lrt(A,\cdot,-)$ as a direct sum of irreducible
$S_3$-modules.
\end{corollary}
\begin{proof}
Any $d\in\sder(A,\cdot,-)$ satisfies $\bar d=-d$ and
\[
d(x\cdot y)=-d(x)\cdot y-x\cdot d(y)
\]
for any $x,y\in A$. With $x=y=1$ it follows that $d(1)=0$, and then
with $y=1$ that $d(x)=-d(x)$ for any $x$, so $d=0$.
\end{proof}

\smallskip

\begin{remark}
Given an algebra with involution $(A,\cdot,-)$, the Lie algebra
$\frl=\lrt(A,\cdot,-)$ decomposes as
\[
\frl=\frl_+\oplus\frl_-\oplus\frl_2,
\]
where $\frl_+$ is the subspace of fixed elements by the action of
any $\sigma\in S_3$ (that is, the sum of the trivial
$S_3$-submodules of $\frl$), $\frl_-$ is the direct sum of the
alternating $S_3$-submodules of $\frl$ and $\frl_2$ is the sum of
the $S_3$-submodules of $\frl$ which are isomorphic to the unique
two-dimensional irreducible module for $S_3$. Then any
$(d_0,d_1,d_2)\in\frl$ decomposes as
\[
(d_0,d_1,d_2)=(d,d,d)+(f,f,f)+(f_0,f_1,f_2),
\]
where $(d,d,d)\in \frl_+$, $(f,f,f)\in\frl_-$ and $(f_0,f_1,f_2)\in
\frl_2$ are given by:
\[
\begin{split}
d&=\frac{1}{6}\bigl((d_0+d_1+d_2)+(\bar d_0+\bar d_1+\bar
d_2)\bigr),\\
f&=\frac{1}{6}\bigl((d_0+d_1+d_2)-(\bar d_0+\bar d_1+\bar
d_2)\bigr),\\
f_i&=\frac{1}{3}(2d_i-d_{i+1}-d_{i+2})\quad\text{(indices modulo
$3$).}
\end{split}
\]
Note that indeed $(1+\varphi+\varphi^2)(f_0,f_1,f_2)=0$, so
$(f_0,f_1,f_2)\in \frl_2$ and, similarly, $(d,d,d)\in\frl_+$ and
$(f,f,f)\in\frl_-$.\hfill\qed
\end{remark}

\bigskip
\section{The Tetrahedron algebra}\label{se:Tet}

The results in the previous section, together with \cite{Edinburgh},
give an immediate answer to \cite[Problem 2.10]{ItoPaul}, which asks
for the decomposition of the Tetrahedron algebra into a direct sum
of irreducible $S_4$-modules.

Recall that the Tetrahedron algebra $\frgtetra$ has been defined in
\cite{HT05} in connection with the so called Onsanger algebra
introduced in \cite{Ons44}, in which the free energy of the two
dimensional Ising model was computed. The Tetrahedron algebra
$\frgtetra$ is the Lie algebra over $k$ with generators
\[
\bigl\{X_{ij}: i,j\in \{0,1,2,3\},\ i\ne j\bigr\}
\]
and relations
\[
\begin{split}
& X_{ij}+X_{ji}=0\ \text{for $i\ne j$,}\\
&[X_{ij},X_{jk}]=2(X_{ij}+X_{jk})\ \text{for mutually distinct
$i,j,k$,}\\
&[X_{hi},[X_{hi},[X_{hi},X_{jk}]]]=4[X_{hi},X_{jk}]\ \text{for
mutually distinct $h,i,j,k$.}
\end{split}
\]

Consider the basis
 $\bigl\{x=\bigl(\begin{smallmatrix}
-1&2\\0&1\end{smallmatrix}\bigr),\,
 y=\bigl(\begin{smallmatrix}
-1&0\\-2&1\end{smallmatrix}\bigr),\,
 z=\bigl(\begin{smallmatrix}
1&0\\0&-1\end{smallmatrix}\bigr)\bigr\}$ of the three dimensional
simple Lie algebra $\frsldos$. Then \cite[Proposition 6.5 and
Theorem 11.5]{HT05} the Tetrahedron algebra is isomorphic to the
three-point loop algebra $\frsl_2\otimes\calA$,
$\calA=k[t,t^{-1},(t-1)^{-1}]$, by means of the isomorphism
\begin{equation}\label{eq:Psi}
\Psi:\frgtetra\rightarrow \frsl_2\otimes\calA,
\end{equation}
determined by:
\[
\begin{split}
\Psi(X_{12})=&x\otimes 1,\quad \Psi(X_{23})=y\otimes 1,\quad
   \Psi(X_{31})=z\otimes 1,\\
&\Psi(X_{03})=y\otimes t +z\otimes (t-1),\\
&\Psi(X_{01})=z\otimes t'+x\otimes (t'-1),\\
&\Psi(X_{02})=x\otimes t''+y\otimes (t''-1),
\end{split}
\]
where
\begin{equation}\label{eq:tes}
t'=1-t^{-1}\quad\text{and}\quad t''=(1-t)^{-1}.
\end{equation}

The symmetric group $S_4$ embeds naturally in the group of
automorphisms $\Aut(\frgtetra)$ by means of
\[
\sigma(X_{ij})=X_{\sigma(i)\sigma(j)},
\]
for any $\sigma\in S_4$ and $0\leq i<j\leq 3$. Here we identify $0$
and $4$. On the other hand, $S_4$ embeds as a group of automorphisms
of $\frg=\frsl_2\otimes \calA$ as follows (\cite[Theorem
1.4]{Edinburgh}):
\begin{romanenumerate}
\item $\varphi=\varphi_\frs\otimes\varphi_\calA$, where
$\varphi_\frs$ is the order $3$ automorphism of $\frsldos$ given by
\[
\varphi_\frs(x)=y,\quad \varphi_\frs(y)=z,\quad \varphi_\frs(z)=x,
\]
and $\varphi_\calA$ is the order $3$ automorphism of the $k$-algebra
$\calA$ determined by
\[
\varphi_\calA(t)=1-t^{-1}=t'.
\]

\item $\tau=\tau_\frs\otimes\tau_\calA$, where $\tau_\frs$ is the
order $2$ automorphism of $\frsldos$ given by
\[
\tau_\frs(x)=-x,\quad \tau_\frs(y)=-z,\quad \tau_\frs(z)=-y,
\]
and $\tau_\calA$ is the order $2$ automorphism of $\calA$ determined
by $\tau_\calA(t)=1-t$.

\item $\tau_1$ is the automorphism of $\frg$, as a Lie algebra over
$\calA$, given by
\begin{equation}\label{eq:tau1}
\begin{split}
\tau_1(x\otimes 1)&=-x\otimes 1,\\
\tau_1(y\otimes 1)&=-\bigl(z\otimes t'+x\otimes (t'-1)\bigr),\\
\tau_1(z\otimes 1)&=x\otimes t''+y\otimes (t''-1).
\end{split}
\end{equation}

\item $\tau_2$ is the automorphism of $\frg$, as a Lie algebra over
$\calA$, given by
\begin{equation}\label{eq:tau2}
\begin{split}
\tau_2(x\otimes 1)&=y\otimes t+z\otimes (t-1),\\
\tau_2(y\otimes 1)&=-y\otimes 1,\\
\tau_2(z\otimes 1)&=-\bigl(x\otimes t''+y\otimes (t''-1)\bigr).
\end{split}
\end{equation}
\end{romanenumerate}

With this action of $S_4$, the isomorphism $\Psi$ in \eqref{eq:Psi}
becomes an isomorphism of Lie algebras with $S_4$-action.

A particular basis $\{u_0,u_1,u_2\}$ of $\frg=\frsl_2\otimes \calA$,
as a module for $\calA$, plays a key role in \cite{Edinburgh}. It is
defined by:
\begin{equation}\label{eq:Abasis}
\begin{split}
u_0&=\frac{1}{4}\Psi(X_{02}+X_{31}),\\
u_1&=\frac{1}{4}\Psi(X_{03}+X_{12}),\\
u_2&=\frac{1}{4}\Psi(X_{01}+X_{23}),
\end{split}
\end{equation}
and satisfies \cite[Theorem 1.9]{Edinburgh}:
\begin{equation}\label{eq:uiuj}
[u_0,u_1]=-u_2t,\quad [u_1,u_2]=-u_0t',\quad [u_2,u_0]=-u_1t''.
\end{equation}
Moreover, these elements also generate $\frg$ as a Lie algebra over
$k$. The action of $S_4$ on this $\calA$-basis is given by:
\[
\begin{split}
&\tau_1(u_0)=u_0=-\tau_2(u_0),\\
&\tau_1(u_1)=-u_1=-\tau_2(u_1),\\
&\tau_1(u_2)=-u_2=\tau_2(u_2),\quad\text{(see \cite[Theorem 2.2]{Edinburgh})}\\
&\varphi(u_i)=u_{i+1},\ \text{(indices modulo $3$, see
\cite[(1.8)]{Edinburgh})}\\
&\tau(u_1)=\frac{1}{4}\Psi(X_{03}-X_{12})=-[u_0,u_1]=u_2t,\
\text{(see \cite[Theorem 1.9]{Edinburgh})}\\
&\tau(u_2)=\tau\varphi(u_1)=\varphi^2\tau(u_1)=u_1t'',\\
&\tau(u_0)=\tau\varphi^2(u_1)=\varphi\tau(u_1)=u_0t'.
\end{split}
\]

Besides, the subspaces $\frt$, $\frg_0$, $\frg_1$ and $\frg_2$ in
\eqref{eq:tg0g1g2} become:
\[
\frt=0,\quad \frg_0=u_0\calA,\quad \frg_1=u_1\calA,\quad
\frg_2=u_2\calA,
\]
and the associated normal LRTA can be identified to
$(\calA,\cdot,-,\delta)$ with (see \cite[Proposition
3.1]{Edinburgh}):
\[
\begin{split}
a\cdot
b&=\bigl(\tau_\calA\varphi_\calA(a)\bigr)(\tau_\calA\varphi_\calA^2(b)\bigr),\\
\bar a&=-t'\tau_\calA(a).
\end{split}
\]
(Here, the usual multiplication of $\calA=k[t,t^{-1},(1-t)^{-1}]$ is
denoted by juxtaposition.)

\smallskip

Then Corollary \ref{co:g0g1g2}, together with $\frt=0$, gives the answer to \cite[Problem 2.10]{ItoPaul}:

\begin{proposition}
The Tetrahedron algebra is a direct sum of copies of the two
irreducible three-dimensional modules for $S_4$.
\end{proposition}

\smallskip

But a more precise statement can be given, by providing a concrete
decomposition into a direct sum of irreducible $S_4$-modules. To do
so, and because of Theorem \ref{th:g0g1g2}, let us first compute the
elements $a\in\calA$ with $\bar a=\pm a$. Note that for $a\in\calA$,
\[
\begin{split}
\bar a=a\  &\Leftrightarrow\ a=-t'\tau_\calA(a)\\
 &\Leftrightarrow\ ta=\tau_\calA(ta),\ \text{as $tt't''=-1$ and
 $\tau_\calA(t)=(t'')^{-1}$.}
\end{split}
\]
But, as $\calA=k[t,t^{-1},(1-t)^{-1}]\subseteq k(t)$, any $b\in
\calA$ is written uniquely as $b=\frac{p(t)}{(t(1-t))^s}$, with
$p(t)$ a polynomial in $k[t]$ with either $p(0)\ne 0$ or $p(1)\ne
0$. Since $\tau_\calA(t)=1-t$, $\tau_\calA(b)=b$ if and only if
$\tau_\calA\bigl(p(t)\bigr)=p(t)$ or, equivalently, $p(t)=p(1-t)$.
This is equivalent to the condition $p(t)\in
k[t(1-t)]=k[\bigl(t-\frac{1}{2}\bigr)^2]$. Note that the set
$\{(t(1-t))^r: r\in \bN\cup\{0\}\}$ is a $k$-basis of the polynomial
ring $k[t(1-t)]$.

In a similar vein, $\tau_\calA(b)=-b$ if and only if $p(t)\in
(2t-1)k[t(1-t)]$. Hence, for any $a\in\calA$:
\[
\begin{split}
\bar a=a\ &\Leftrightarrow\ a\in t^{-1}k[t(1-t),(t(1-t))^{-1}],\\
\bar a=-a\ &\Leftrightarrow\ a\in
t^{-1}(2t-1)k[t(1-t),(t(1-t))^{-1}].
\end{split}
\]

Therefore, Theorem \ref{th:g0g1g2} gives:

\begin{theorem}\label{th:tet33bar}
For any integer $s$,
\begin{romanenumerate}
\item
$V_s'=\espan{u_0t^{-1}(t(1-t))^s,u_1(t')^{-1}(t'(1-t'))^s,u_2(t'')^{-1}(t''(1-t''))^s}$
is isomorphic, as a module for $S_4$, to $V'=U'\otimes V$.
\item $V_s=\text{span}\left\{u_0t^{-1}(2t-1)(t(1-t))^s,u_1(t')^{-1}(2t'-1)(t'(1-t'))^s
,\right.$
\newline $\left.u_2(t'')^{-1}(2t''-1)(t''(1-t''))^s\right\}$ is
isomorphic, as a module for $S_4$, to the standard module $V$.
\item The decomposition
\[
\frg=
\Bigl(\oplus_{s\in\bZ}V_s'\Bigr)\oplus\Bigl(\oplus_{s\in\bZ}V_s\Bigr)
\]
is a decomposition of $\frg=\frsl_2\otimes \calA$ into a direct sum
of irreducible $S_4$-modules.
\end{romanenumerate}
\end{theorem}
\begin{proof}
(i) and (ii) are direct consequences of Theorem \ref{th:g0g1g2}.
Also, as $u_0$, $u_1$ and $u_2$ form a basis of $\frg$ over $\calA$,
and
\[
\calA=k[t,t^{-1},(1-t)^{-1}]=k[t',(t')^{-1},(1-t')^{-1}]
=k[t'',(t'')^{-1},(1-t'')^{-1}],
\]
the result in (iii) follows.
\end{proof}

\smallskip

\begin{remark}
The symmetric group $S_3$ acts on $\calA$ with the action of
$\varphi$ and $\tau$ given by $\varphi_\calA$ and $\tau_\calA$
above. As noticed in \cite[Lemma 6.3]{HT05},
$\{1\}\cup\{t^n,(t')^n,(t'')^n: n\in\bN\}$ is a $k$-basis of
$\calA$. Now, recall that $\varphi_\calA$ permutes cyclically $t$,
$t'$ and $t''$, while $\tau_\calA(t)=1-t$, so
$\tau_\calA(2t-1)=-(2t-1)$ and
$\tau_\calA(2t'-1)=\tau_\calA\varphi_\calA(2t-1)=
\varphi_\calA^2\tau_\calA(2t-1)=-(2t''-1)$. Thus, for any $n\in\bN$:
\begin{itemize}
\item $W_n=\{ \alpha(2t-1)^n+\alpha'(2t'-1)^n+\alpha''(2t''-1)^n:
\alpha,\alpha',\alpha''\in k,\, \alpha+\alpha'+\alpha''=0\}$ is a
two dimensional irreducible module for $S_3$,
\item
$U_n=k\Bigl((2t-1)^{2(n-1)}+(2t'-1)^{2(n-1)}+(2t''-1)^{2(n-1)}\Bigr)$
is a trivial module for $S_3$,
\item
$U_n'=k\Bigl((2t-1)^{2n-1}+(2t'-1)^{2n-1}+(2t''-1)^{2n-1}\Bigr)$ is
an irreducible module for $S_3$ isomorphic to the alternating
module,
\end{itemize}
and there appears the following decomposition into a direct sum of
irreducible modules:
\[
\null\qquad\qquad \calA=\oplus_{n\in\bN}(U_n\oplus U_n'\oplus W_n).\qquad\qquad\qed
\]
\end{remark}

\bigskip

The symmetric group $S_4$ acts on the set of generators
$\bigl\{ X_{ij}: i,j\in\{0,1,2,3\},\, i\ne j\bigr\}$ of the
Tetrahedron algebra. The subspace spanned by these generators is
easily shown to split into the direct sum of two nonisomorphic
irreducible $S_4$-modules (recall that the index $0$ is identified
to $4$):
\[
\begin{split}
&\espan{X_{\sigma(1)\sigma(2)}+X_{\sigma(2)\sigma(3)}+X_{\sigma(3)\sigma(4)}+
X_{\sigma(4)\sigma(1)}: \sigma\in S_4}\\
&\qquad\qquad =k(X_{01}+X_{12}+X_{23}+X_{30})+k(X_{02}+X_{23}+X_{31}+X_{10})\\
&\qquad\qquad\qquad\qquad +k(X_{03}+X_{31}+X_{12}+X_{20}),
\end{split}
\]
and
\[
\begin{split}
&\espan{X_{\sigma(1)\sigma(2)}-X_{\sigma(2)\sigma(3)}+X_{\sigma(3)\sigma(4)}-
X_{\sigma(4)\sigma(1)}: \sigma\in S_4}\\
&\qquad\qquad =k(X_{01}-X_{12}+X_{23}-X_{30})+k(X_{02}-X_{23}+X_{31}-X_{10})\\
&\qquad\qquad\qquad\qquad +k(X_{03}-X_{31}+X_{12}-X_{20}).
\end{split}
\]

Note that
\[
\begin{split}
\frac{1}{4}\Psi\bigl(X_{01}+X_{12}+X_{23}+X_{30}\bigr)&=
\frac{1}{4}\Psi\bigl(X_{01}+X_{23}\bigr)
+\frac{1}{4}\Psi\bigl(X_{12}-X_{03}\bigr)\\
&=u_2-u_2t=u_2(1-t)=u_2(t'')^{-1}.
\end{split}
\]
(See the proof of \cite[Theorem 1.9]{Edinburgh}.) By applying
$\varphi$ and the fact that $\Psi$ is an isomorphism that commutes
with the actions of $S_4$, we obtain:
\[
\begin{split}
\frac{1}{4}\Psi\bigl(X_{01}+X_{12}+X_{23}+X_{30}\bigr)&=u_2(t'')^{-1}=:v_2,\\
\frac{1}{4}\Psi\bigl(X_{02}+X_{23}+X_{31}+X_{10}\bigr)&=u_0t^{-1}=:v_0,\\
\frac{1}{4}\Psi\bigl(X_{03}+X_{31}+X_{12}+X_{20}\bigr)&=u_1(t')^{-1}=:v_1,
\end{split}
\]
and, similarly,
\[
\begin{split}
\frac{1}{4}\Psi\bigl(X_{01}-X_{12}+X_{23}-X_{30}\bigr)&=u_2(1+t)=:w_2,\\
\frac{1}{4}\Psi\bigl(X_{02}-X_{23}+X_{31}-X_{10}\bigr)&=u_0(1+t')=:w_0,\\
\frac{1}{4}\Psi\bigl(X_{03}-X_{31}+X_{12}-X_{20}\bigr)&=u_1(1+t'')=:w_1.
\end{split}
\]

Either by looking at Theorem \ref{th:tet33bar} or by direct
computation, it follows that $kv_0+kv_1+kv_2$ is isomorphic to the
$S_4$-module $V'$, and that $kw_0+kw_1+kw_2$ is isomorphic to the
standard module $V$. Since $1+t$, $1+t'$ and $1+t''$ are not
invertible in $\calA$, $\{w_0,w_1,w_2\}$ is not a basis of $\frg$
over $\calA$. On the other hand, $\{v_0,v_1,v_2\}$ is indeed a basis
of $\frg=\frsl_2\otimes\calA$ over $\calA$. However, $v_0$, $v_1$
and $v_2$ fail to generate the whole Lie algebra $\frg$, as the
following two results show.

\begin{proposition}
The $k$-subalgebra generated by $v_0$, $v_1$ and $v_2$ is
$v_0\calS\oplus v_1\calS\oplus v_2\calS$, where
$\calS=k[t(1-t),t'(1-t'),t''(1-t'')]$.
\end{proposition}
\begin{proof}
Since equation \eqref{eq:uiuj} shows that:
\begin{equation}\label{eq:v1v2}
\begin{split}
[v_1,v_2]&=[u_1(t')^{-1},u_2(t'')^{-1}]=-u_0t'(t')^{-1}(t'')^{-1}\\
 &=-u_0t^{-1}t(t'')^{-1}=-v_0t(1-t)\in v_0\calS,
\end{split}
\end{equation}
and by applying $\varphi$ one obtains too:
\begin{equation}\label{eq:v2v0}
[v_2,v_0]=-v_1t'(1-t')\in v_1\calS,
\end{equation}
\begin{equation}\label{eq:v0v1}
[v_0,v_1]=-v_2t''(1-t'')\in v_2\calS,
\end{equation}
it follows that $v_0\calS\oplus v_1\calS\oplus v_2\calS$ is a
subalgebra of $\frg$, and therefore, the subalgebra generated by
$v_0$, $v_1$ and $v_2$ is contained in $v_0\calS\oplus
v_1\calS\oplus v_2\calS$.

Conversely, as
\[
\begin{split}
t(1-t)&\xrightarrow{\varphi_\calA}
t'(1-t')\xrightarrow{\varphi_\calA} t''(1-t''),\\
t(1-t)&\xrightarrow{\tau_\calA} t(1-t),
\end{split}
\]
it follows that $\calS$ is invariant under the action of $S_3$.
Also, $kv_0+kv_1+kv_2$ is invariant under the action of $S_4$, so
the algebra generated by the $v_i$'s: $\alg\langle
v_0,v_1,v_2\rangle$, is invariant too under the action of $S_4$. In
particular, the action of Klein's $4$-group shows that $\alg\langle
v_0,v_1,v_2\rangle =v_0\calS_0\oplus v_1\calS_1\oplus v_2\calS_2$,
where $\calS_0$, $\calS_1$ and $\calS_2$ are subspaces of $\calA$
containing the unity $1$. Also, the invariance under the action of
$\varphi$ shows that (since $\varphi(v_i)=v_{i+1}$)
$\varphi_\calA(\calS_i)\subseteq \calS_{i+1}$ for $i=0,1,2$.

Equations \eqref{eq:v1v2}, \eqref{eq:v2v0} and \eqref{eq:v0v1} give
\begin{equation}\label{eq:SiSj}
t(1-t)\calS_1\calS_2\subseteq \calS_0,\quad
t'(1-t')\calS_2\calS_0\subseteq \calS_1,\quad
t''(1-t'')\calS_0\calS_1\subseteq \calS_2.
\end{equation}
But
\[
\begin{split}
[v_1,[v_1,v_2]]&=[v_0,v_1]t(1-t)=-v_2t(1-t)t''(1-t'')=-v_2t(1-t'')\\
 &=-v_2t(t')^{-1}=-v_2\frac{1}{t'(1-t')},
\end{split}
\]
as $t''=(1-t)^{-1}$ and cyclically. Therefore
$\frac{1}{t'(1-t')}\calS_1^2\calS_2$ is contained in $\calS_2$. In
particular, since $1$ is in contained in the $\calS_i$'s, one has:
\begin{equation}\label{eq:SicSj}
\frac{1}{t'(1-t')}\calS_1\subseteq \calS_2,\quad
\frac{1}{t'(1-t')}\calS_2\subseteq \calS_2.
\end{equation}

Equations \eqref{eq:SiSj} and \eqref{eq:SicSj} give
$t'(1-t')\calS_2\subseteq \calS_1\subseteq t'(1-t')\calS_2$, so
$\calS_1=t'(1-t')\calS_2$ and
\[
\calS_2=t'(1-t')\frac{1}{t'(1-t')}\calS_2\subseteq
t'(1-t')\calS_2=\calS_1.
\]
Thus, $\calS_2\subseteq \calS_1$ and, by applying $\varphi_\calA$
one obtains $\calS_2\subseteq \calS_1\subseteq \calS_0\subseteq
\calS_2$. Therefore, $\calS_0=\calS_1=\calS_2$ and also
$t'(1-t')\calS_2= \calS_2$. Now \eqref{eq:SiSj} shows that
$\calS_2$ is a subalgebra of $\calA$, and $\calS_2$ contains the
element $t'(1-t')$, so it contains too, by invariance under the
action of $S_3$, the elements $t(1-t)$ and $t''(1-t'')$, which shows
that $\calS_0=\calS_1=\calS_2\supseteq \calS$, as required.
\end{proof}

\begin{proposition}
$\calS$ is a maximal subalgebra of $\calA$.
\end{proposition}
\begin{proof}
Recall that $\calS=k[t(1-t),t'(1-t'),t''(1-t'')]$. Since
\[
t(1-t)t'(1-t')t''(1-t'')=t(t'')^{-1}t't^{-1}t''(t')^{-1}=1,
\]
it follows that $(t(1-t))^{-1}$, $(t'(1-t'))^{-1}$ and
$(t''(1-t''))^{-1}$ are all contained in $\calS$ and that the
subspace of $\calA$ fixed by $\tau_\calA$, which is
$\calB=k[t(1-t),(t(1-t))^{-1}]$ is contained in $\calS$. Hence
\[
\calS=\calB\oplus\{s\in\calS:\tau_\calA(s)=-s\}.
\]
Since $\calA$ is an integral domain, the eigenspace of eigenvalue
$-1$ for $\tau_\calA$ is the free $\calB$-module $\calB(2t-1)$, so
that $\{s\in\calS: \tau_\calA(s)=-s\}$ is a $\calB$-submodule of
$\calB(2t-1)$. Let us check that $\{s\in\calS:\tau_\calA(s)=-s\}$
coincides with $\calB(2t-1)(1-t(1-t))$. Since
$\calB(2t-1)=\calB(2t-1)(1-t(1-t))\oplus k(2t-1)$, this will show
that $\calS$ is a codimension one subalgebra of $\calA$, and hence
maximal.

Note that
\[\begin{split}
(t-1)^3&=(1-2t+t^2)(t-1)=(t-1)+(2-t)t(1-t)\\
  &=\frac{1}{2}\Bigl((2t-1)-1+(3-(2t-1))t(1-t)\Bigr)\\
  &=\frac{1}{2}\Bigl((2t-1)(1-t(1-t))-1+3t(1-t)\Bigr),\\[4pt]
t^3&=(t-1)^3+3t^2-3t+1=(t-1)^3-3t(1-t)+1\\
  &=\frac{1}{2}\Bigl((2t-1)(1-t(1-t))+1-3t(1-t)\Bigr),
\end{split}
\]
and hence
\[
\begin{split}
t'(1-t')&=t't^{-1}=(1-t^{-1})t^{-1}=\frac{t-1}{t^2}=\frac{(t-1)^3}{(t(1-t))^2}\\
&=\frac{1}{2(t(1-t))^2}\Bigl((2t-1)(1-t(1-t))-1+3t(1-t)\Bigr),\\[8pt]
t''(1-t'')&=t''(t')^{-1}=-\frac{t}{(1-t)^2}=-\frac{t^3}{(t(1-t))^2}\\
 &=\frac{-1}{2(t(1-t))^2}\Bigl((2t-1)(1-t(1-t))+1-3t(1-t)\Bigr),
\end{split}
\]
so $\calS=k[t(1-t),t'(1-t'),t''(1-t'')]\subseteq
\calB\oplus\calB(2t-1)(1-t(1-t))$ holds.

Reciprocally, $\calB$ is contained in $\calS$, and
$(2t-1)(1-t(1-t))=2t'(1-t')(t(1-t))^2+1-3t(1-t)$
belongs to $\calS$. Hence we have $\calS=\calB\oplus\calB(2t-1)(1-t(1-t))$, as required.
\end{proof}

\bigskip
\section{Lie algebras with $S_3$-action}\label{se:S3}

Given any normal LRTA, the Lie algebra
$\tilde\frg(A,\cdot,-,\delta)$ in \eqref{eq:tildefrg} is endowed
with an action of the symmetric group $S_4$ with its Lie algebra of
related triples being fixed by Klein's $4$-group. Thus
$\lrt(A,\cdot,-)$ is endowed naturally with an action of $S_3$,
considered in Proposition \ref{pr:S3lrtA}.

This section is devoted
to general Lie algebras endowed with an action of $S_3$ by
automorphisms.

First of all, consider the irreducible $S_3$-modules:
\begin{itemize}
\item $U=ku$, the trivial module: $\sigma(u)=u$ for any $\sigma\in
S_3$,
\item $U'=ku'$, the alternating module: $\sigma(u')=(-1)^\sigma u'$
for any $\sigma\in S_3$,
\item $W=\{(\alpha,\beta,\gamma)\in k^3: \alpha+\beta+\gamma=0\}$,
the two dimensional irreducible module. Here $W=kw_+ +kw_-$, with
$w_+=(-1,-1,2)$ and $w_-=(3,-3,0)$.
\end{itemize}

Note that there are the following obvious isomorphisms of
$S_3$-modules
\[U\otimes M\simeq M:\ u\otimes m\leftrightarrow m
\]
for any
$S_3$-module $M$.

Also, as $S_3$-modules, $U$ can be identified to $k(1,1,1)\subseteq
k^3$, so $U\oplus W=k^3$ (with the natural action
$\sigma\bigl((\alpha_1,\alpha_2,\alpha_3)\bigr)=(\alpha_{\sigma^{-1}(1)},
\alpha_{\sigma^{-1}(2)},\alpha_{\sigma^{-1}(3)})$ for any
$\alpha_1,\alpha_2,\alpha_3\in k$). The componentwise multiplication
in $k^3$ gives the following homomorphisms of $S_3$-modules (with
$u=(1,1,1)$):
\begin{equation}\label{eq:WWWU}
\begin{aligned}
W\otimes W&\rightarrow W\\ x\otimes y\,&\mapsto x\bullet
y=\pi_W(xy),
\end{aligned}
\qquad
\begin{aligned}
W\otimes W&\rightarrow U=ku\\
x\otimes y\,&\mapsto (x\,\vert\, y)u=\pi_U(xy),
\end{aligned}
\end{equation}
where $\pi_W$ and $\pi_U$ denote the projections of $k^3$ onto $W$
and $U$ respectively. Thus, $\bullet$ is a commutative product with
\begin{equation}\label{eq:wsbullet}
w_+\bullet w_+=w_+,\qquad w_+\bullet w_-=-w_-,\qquad w_-\bullet
w_-=-3w_+.
\end{equation}
Note that $(W,\bullet)$ is the para-Hurwitz algebra attached to the
two dimensional composition algebra $k[z]=k1+kz$ with $z^2=-3$. That
is, $x\bullet y=\overline{xy}$ for any $x,y\in k[z]$, where $\bar x$
denotes the natural involution of $k[z]$, given by $\bar z=-z$. This
will be exploited in Section \ref{se:S3examples}.

Also, $(.\,\vert\, .)$ is a nondegenerate symmetric bilinear form with:
\begin{equation}\label{eq:wsbilinear}
(w_+\,\vert\, w_+)=2,\qquad (w_-\,\vert\, w_-)=6,\qquad (w_+\,\vert\, w_-)=0.
\end{equation}

Moreover, the $S_3$-modules $\wedge^2 W$ and $U'$ are isomorphic, so
with $U'=ku'$, there is the homomorphism of $S_3$-modules
\begin{equation}\label{eq:WWUprime}
\begin{split}
W\otimes W&\rightarrow U'\\
x\otimes y\,&\mapsto \langle x\,\vert\, y\rangle u',
\end{split}
\end{equation}
with $\langle .\,\vert\, .\rangle$ the alternating bilinear form with
$\langle w_-\,\vert\, w_+\rangle =1$. Since $W\otimes W$ is isomorphic,
as a $S_3$-module, to $U\oplus U'\oplus W$, the homomorphisms in
\eqref{eq:WWWU} and \eqref{eq:WWUprime} are, up to scalars, the only
possible such homomorphisms. Besides, there are unique, up to
scalars, Clebsch-Gordan homomorphisms of $S_3$-modules $\diamond:U'\otimes
U'\rightarrow U$ and $\diamond:U'\otimes W\rightarrow W$ given by:
\begin{equation}\label{eq:UprimeUprime}
\begin{split}
u'\otimes u'&\mapsto u'\diamond u'=-12 u,\\
u'\otimes w_+&\mapsto u'\diamond w_+=2w_-,\\
u'\otimes w_-&\mapsto u'\diamond w_-=-6w_+.
\end{split}
\end{equation}

\medskip

Let $\frg$ be a Lie algebra endowed with an action of $S_3$ by
automorphisms. Then by complete reducibility, $\frg$ is a direct sum
of copies of $U$, $U'$ and $W$. By joining together isomorphic
summands, we may write:
\begin{equation}\label{eq:gd+d-WM}
\frg=(U\otimes \frd^+)\oplus(U'\otimes \frd^-)\oplus (W\otimes M),
\end{equation}
for vector spaces $\frd^+$, $\frd^-$ and $M$.

Note that $(U\otimes\frd^+)\oplus(U'\otimes \frd^-)$ is the
subalgebra of $\frg$ consisting of the elements fixed by the action
of $\varphi=(123)$, and it is graded over $\bZ_2$ by means of the
action of $\tau=(12)$, which is the identity on $U\otimes\frd^+$ and
minus the indentity on $U'\otimes\frd^-$. The remaining brackets in
$\frg$ are given, by $S_3$-invariance, by:
\begin{equation}\label{eq:bracketS3}
\begin{split}
&[u\otimes d_1^+,u\otimes d_2^+]=u\otimes[d_1^+,d_2^+],\\
&[u\otimes d^+,u'\otimes d^-]=u'\otimes [d^+,d^-],\\
&[u'\otimes d_1^-,u'\otimes d_2^-]=(u'\diamond u')\otimes
[d_1^-,d_2^-]
   =-12u\otimes [d_1^-,d_2^-],\\
&[u\otimes d^+,w\otimes x]=w\otimes d^+(x),\\
&[u'\otimes d^-,w\otimes x]=(u'\diamond w)\otimes d^-(x),\\
&[w_1\otimes x,w_2\otimes y]=(w_1\,\vert\, w_2)u\otimes d_{x,y}^+
   +\langle w_1\,\vert\, w_2\rangle u'\otimes d_{x,y}^- +w_1\bullet
   w_2\otimes xy,
\end{split}
\end{equation}
for any $d^{\pm},d_1^{\pm},d_2^\pm\in \frd^{\pm}$, $x,y\in M$ and
$w,w_1,w_2\in W$, where
\[
\begin{split}
&\frd^\pm\times \frd^\pm\rightarrow \frd^+,\quad
(d_1^\pm,d_2^\pm)\mapsto [d_1^\pm,d_2^\pm],\\
&\frd^+\times \frd^-\rightarrow \frd^-,\quad (d^+,d^-)\mapsto
[d^+,d^-],\\
&\frd^\pm\times M\rightarrow M,\quad (d^\pm,x)\mapsto d^\pm(x),\\
&M\times M\rightarrow \frd^\pm,\quad (x,y)\mapsto d_{x,y}^\pm,\\
&M\times M\rightarrow M,\quad (x,y)\mapsto xy,
\end{split}
\]
are bilinear maps. The skew-symmetry of the bracket in $\frg$
forces $[.,.]$ to be skew-symmetric on its arguments, $d_{x,y}^+$
and $xy$ to be skew-symmetric too, and $d_{x,y}^-$ to be symmetric
on $x,y$.

In this situation consider the trilinear map:
\begin{equation}\label{eq:tripleproduct}
\begin{split}
\{.,.,.\}: M\times M\times M&\rightarrow M\\
(x,y,z)&\mapsto \{x,y,z\}=d_{x,y}^+(z)+d_{x,y}^-(z).
\end{split}
\end{equation}

The next result shows that, essentially, the Lie algebra $\frg$ is determined by the triple $\bigl(M,xy,\{x,y,z\}\bigr)$.

\begin{theorem}\label{th:S3}
Let $\frg$ be a Lie algebra endowed with an action of $S_3$ by
automorphisms, so that $\varphi$ acts nontrivially. With the
notations above, for any $a,b,x,y,z,t\in M$:
\begin{subequations}\label{eq:identitiesS3}
\begin{equation}\label{eq:idS3-4}
(xy)z=\{x,z,y\}-\{y,z,x\}.
\end{equation}
\begin{equation}\label{eq:idS3-1}
\{a,b,xy\}=-\{b,a,x\}y-x\{b,a,y\},
\end{equation}
\begin{equation}\label{eq:idS3-2}
\{a,b,\{x,y,z\}\}-\{x,y,\{a,b,z\}\}=\{\{a,b,x\},y,z\}-\{x,\{b,a,y\},z\},
\end{equation}
(that is, $(M,\{.,.,.\})$ is a generalized Jordan triple system
\cite{Kantor})
\begin{equation}\label{eq:idS3-3}
\{xy,z,t\}+\{yz,x,t\}+\{zx,y,t\}=0=\{x,yz,t\}+\{y,zx,t\}+\{z,xy,t\},
\end{equation}
\end{subequations}

Conversely, let $(M,xy,\{x,y,z\})$ be a nonzero vector space endowed
with an anticommutative bilinear product $xy$ and a trilinear
product $\{x,y,z\}$ satisfying the equations
\eqref{eq:identitiesS3}. Consider the vector space
\begin{equation}\label{eq:gM}
\frg(M)=(U\otimes \frd^+)\oplus(U'\otimes \frd^-)\oplus (W\otimes
M),
\end{equation}
with
\[
\frd^+=\espan{\{x,y,.\}-\{y,x,.\}:x,y\in M}
\]
and
\[
\frd^-=\espan{\{x,y,.\}+\{y,x,.\}:x,y\in M}
\]
(thus, $\frd^+$ and $\frd^-$ are subspaces of the general Lie algebra $\frgl(M)$). Then
the subspace $\frd^+ + \frd^-=\espan{\{x,y,.\}: x,y\in M}$ is a
Lie subalgebra of $\frgl(M)$, and the vector space $\frg(M)$, with
the bracket defined by \eqref{eq:bracketS3}, where
$d_{x,y}^+=\frac{1}{2}(\{x,y,.\}-\{y,x,.\})$ and
$d_{x,y}^-=\frac{1}{2}(\{x,y,.\}+\{y,x,.\})$, is a Lie algebra with
an action of $S_3$ by automorphisms defined by the action of $S_3$
on $U$, $U'$ and $W$.
\end{theorem}
\begin{proof}
Without loss of generality we may extend scalars if necessary and
assume that $\omega\in k$, where $\omega^3=1\ne \omega$. Then $W$
contains the basis
$\{w_\omega=(\omega^2,\omega,1),w_{\omega^2}=(\omega,\omega^2,1)\}$,
whose elements satisfy
\[
\varphi(w_\omega)=\omega w_\omega,\quad
\varphi(w_{\omega^2})=\omega^2w_{\omega^2},\quad
\tau(w_\omega)=w_{\omega^2}.
\]
Also,
\[
w_\omega+w_{\omega^2}=w_+\quad\text{and}\quad
w_\omega-w_{\omega^2}=\frac{\omega^2-\omega}{3}w_-,
\]
so
\[
w_\omega=\frac{1}{2}\bigl(w_+ +\frac{\omega^2-\omega}{3}w_-\bigr),\quad
w_{\omega^2}=\frac{1}{2}\bigl(w_+
-\frac{\omega^2-\omega}{3}w_-\bigr),
\]
and by equations \eqref{eq:wsbullet} and \eqref{eq:wsbilinear}:
\[
\begin{split}
&w_\omega\bullet w_\omega=w_{\omega^2},\quad w_{\omega^2}\bullet
w_{\omega^2}=w_\omega,\quad w_\omega\bullet w_{\omega^2}=0,\\[2pt]
&(w_\omega\,\vert\, w_\omega)=0=(w_{\omega^2}\,\vert\, w_{\omega^2}),\quad
(w_\omega\,\vert\, w_{\omega^2})=1,
\end{split}
\]
and
\[
\langle w_\omega \,\vert\, w_{\omega^2}\rangle =\frac{1}{4}\langle
w_+ +\frac{\omega^2-\omega}{3}w_- \,\vert\, w_+
-\frac{\omega^2-\omega}{3}w_-\rangle=\frac{\omega^2-\omega}{6}.
\]
Also,
\begin{equation}\label{eq:uprimewomega}
\begin{split}
u'\diamond w_\omega&=\frac{1}{2}u'\diamond \bigl(w_+
   +\frac{\omega^2-\omega}{3}w_-\bigr)\\
  & =w_-+(\omega-\omega^2)w_+=2(\omega-\omega^2)w_\omega\\[6pt]
u'\diamond w_{\omega^2}&=\frac{1}{2}u'\diamond \bigl(w_+
   -\frac{\omega^2-\omega}{3}w_-\bigr)\\
   &=w_-+(\omega^2-\omega)w_+=-2(\omega-\omega^2)w_{\omega^2}.
\end{split}
\end{equation}
Take $\hat
u'=\frac{1}{2(\omega-\omega^2)}u'=\frac{\omega^2-\omega}{6}u'$. Then
\begin{equation}\label{eq:uhatprimewomega}
\begin{split}
&\hat u'\diamond w_\omega=w_\omega,\quad
 \hat u'\diamond w_{\omega^2}=-w_{\omega^2},\\[6pt]
 &\hat u'\diamond\hat u'=-12\frac{1}{(2(\omega-\omega^2))^2}u=u.
\end{split}
\end{equation}
Consider also the new alternating bilinear form $\langle .\,\vert\,
.\rangle\hat{}{\,}:W\times W\rightarrow k$ given by $\langle
w_\omega\,\vert\, w_{\omega^2}\rangle\hat{}{\,}=1$, that is, $\langle
.\,\vert\, .\rangle\hat{}{\,}=\frac{6}{\omega^2-\omega}\langle .\,\vert\,
.\rangle$.

Then the brackets in \eqref{eq:bracketS3} become simpler:
\begin{equation}\label{eq:bracketS3bis}
\begin{split}
&[u\otimes d_1^+,u\otimes d_2^+]=u\otimes [d_1^+,d_2^+],\\
&[u\otimes d^+,\hat u'\otimes d^-]=\hat u'\otimes [d^+,d^-],\\
&[\hat u'\otimes d_1^-,\hat u'\otimes d_2^-]=u\otimes [d_1^-,d_2^-],\\
&[u\otimes d^+,w\otimes x]=w\otimes d^+(x),\\
&[\hat u',w\otimes x]=(\hat u'\diamond w)\otimes d^-(x),\\
&[w_1\otimes x,w_2\otimes y]=
 (w_1\,\vert\, w_2)u\otimes d_{x,y}^++\langle w_1\,\vert\,
 w_2\rangle\hat{}{\,}\hat u'\otimes d_{x,y}^-+ w_1\bullet w_2\otimes
 xy,
\end{split}
\end{equation}
for any $d^\pm,d_1^\pm,d_2^\pm\in\frd^\pm$, $x,y\in M$ and
$w_1,w_2\in W$. This shows in particular that the vector space
$\frd=\frd^+\oplus\frd^-$ is a $\bZ_2$-graded Lie algebra (with even
part $\frd^+$ and odd part $\frd^-$) with the bracket given by the
brackets $[d_1^\pm,d_2^\pm]$.

So assume that $\frg$ is a Lie algebra with an action of $S_3$ by
automorphisms, decomposition as in \eqref{eq:gd+d-WM} and Lie bracket as
in \eqref{eq:bracketS3bis}. Then for any $x,y,z\in M$,
\[
\begin{split}
[[w_{\omega^2}\otimes x,w_{\omega^2}\otimes y],w_{\omega^2}\otimes
z]&=[w_\omega\otimes xy,w_{\omega^2}\otimes z]\\
&=u\otimes d_{xy,z}^++\hat u'\otimes d_{xy,z}^-.
\end{split}
\]
The Jacobi identity then shows that
\[
d_{xy,z}^\pm+d_{yz,x}^\pm+d_{zx,y}^\pm=0,
\]
and this proves \eqref{eq:idS3-3}. Now,
\[
\begin{split}
[[w_{\omega^2}\otimes x,w_{\omega^2}\otimes y],w_\omega\otimes z]&=
 [w_\omega\otimes xy,w_\omega\otimes
 z]=w_{\omega^2}\otimes(xy)z,\\[6pt]
[[w_{\omega^2}\otimes y,w_\omega\otimes z],w_{\omega^2}\otimes x]&=
 [u\otimes d_{y,z}^+-\hat u'\otimes d_{y,z}^-,w_{\omega^2}\otimes
 x]\\
 &=w_{\omega^2}\otimes d_{y,z}^+(x)-(\hat u'\diamond
 w_{\omega^2})\otimes d_{y,z}^-(x)\\
 &=w_{\omega^2}\otimes \bigl(d_{y,z}^+(x)+d_{y,z}^-(x)\bigr)\\
 &=w_{\omega^2}\otimes \{y,z,x\},\\[6pt]
[[w_\omega\otimes z,w_{\omega^2}\otimes x],w_{\omega^2}\otimes y]&=
 [u\otimes d_{z,x}^++\hat u'\otimes d_{z,x}^-,w_{\omega^2}\otimes
 y]\\
 &=w_{\omega^2}\otimes d_{z,x}^+(y)+(\hat u'\diamond
 w_{\omega^2})\otimes d_{z,x}^-(y)\\
 &=w_{\omega^2}\otimes\bigl(d_{z,x}^+(y)-d_{z,x}^-(y)\bigr)\\
 &=-w_{\omega^2}\otimes\bigl(d_{x,z}^+(y)+d_{x,z}^-(y)\bigr)\\
 &=-w_{\omega^2}\otimes\{x,z,y\},
\end{split}
\]
since $d_{x,y}^+$ is skew-symmetric on $x,y$ and $d_{x,y}^-$ is
symmetric. Hence the Jacobi identity shows that
$(xy)z+\{y,z,x\}-\{x,z,y\}=0$, which gives \eqref{eq:idS3-4}.

Now, for any $d_1^+,d_2^+\in\frd^+$ and $z\in M$:
\[
\begin{split}
[[u\otimes d_1^+,u\otimes d_2^+],w_\omega\otimes z]&=
 [u\otimes[d_1^+,d_2^+],w_\omega\otimes z]\\
 &=w_\omega\otimes [d_1^+,d_2^+](z),\\[6pt]
[u\otimes d_1^+,[u\otimes d_2^+,w_\omega\otimes z]]&=
 w_\omega\otimes d_1^+(d_2^+(z)),\\[6pt]
[u\otimes d_2^+,[u\otimes d_1^+,w_\omega\otimes z]]&=
 w_\omega\otimes d_2^+(d_1^+(z)).
\end{split}
\]
Hence, the bilinear map $\frd^+\times M\rightarrow M$,
$(d^+,x)\mapsto d^+(x)$ is a representation of the Lie algebra
$\frd^+$. In the same vein, one proves:
\[
\begin{split}
[d^+,d^-](z)&=d^+(d^-(z))-d^-(d^+(z)),\\
[d_1^-,d_2^-](z)&=d_1^-(d_2^-(z))-d_2^-(d_1^-(z)),
\end{split}
\]
for any $d^\pm,d_1^\pm,d_2^\pm\in \frd^\pm$, so the linear map:
\[
\begin{split}
\rho:\frd^+\oplus\frd^-&\rightarrow \frgl(M)\\
d^++d^-&\mapsto \rho(d^+ +d^-):z\mapsto d^+(z)+d^-(z),
\end{split}
\]
is a representation of the Lie algebra $\frd^+\oplus\frd^-$.

Finally, for any $d^+\in\frd^+$, $w_1,w_2\in W$ and $x,y\in M$:
\[
\begin{split}
&[u\otimes d^+,[w_1\otimes x,w_2\otimes y]]\\
 &\ =[u\otimes d^+,(w_1\,\vert\, w_2)u\otimes d_{x,y}^+ +\langle w_1\,\vert\,
 w_2\rangle\hat{}{\,} \hat u'\otimes d_{x,y}^-+ (w_1\bullet w_2)\otimes
 xy]\\
 &\ =(w_1\,\vert\, w_2)u\otimes [d^+,d_{x,y}^+]+\langle w_1\,\vert\,
 w_2\rangle\hat{}{\,}\hat u'\otimes [d^+,d_{x,y}^-]+(w_1\bullet
 w_2)\otimes d^+(xy),\\[6pt]
&[[u\otimes d^+,w_1\otimes x],w_2\otimes y]\\
 &\ =[w_1\otimes d^+(x),w_2\otimes y]\\
 &\ =(w_1\,\vert\, w_2)u\otimes d_{d^+(x),y}^+ +\langle w_1\,\vert\,
 w_2\rangle\hat{}{\,}\hat u'\otimes d_{d^+(x),y}^- +(w_1\bullet
 w_2)\otimes d^+(x)y\\[6pt]
&[w_1\otimes x,[u\otimes d^+,w_2\otimes y]]\\
 &\ =[w_1\otimes x,w_2\otimes d^+(y)]\\
 &\ =(w_1\,\vert\, w_2)u\otimes d_{x,d^+(y)}^++\langle w_1\,\vert\,
 w_2\rangle\hat{}{\,} \hat u'\otimes d_{x,d^+(y)}^- +(w_1\bullet
 w_2)\otimes xd^+(y).
\end{split}
\]
Therefore,
\begin{equation}\label{eq:d+}
d^+(xy)=d^+(x)y +xd^+(y)\quad\text{and}\quad
[d^+,d_{x,y}^\pm]=d_{d^+(x),y}^\pm+d_{x,d^+(y)}^\pm,
\end{equation}
for any $x,y\in M$ and $d^+\in\frd^+$. In a similar vein, for any
$d^-\in\frd^-$, $w_1,w_2\in W$ and $x,y\in M$:
\[
\begin{split}
[\hat u'&\otimes d^-,[w_1\otimes x,w_2\otimes y]]\\
 &=[\hat u'\otimes d^-,(w_1\,\vert\, w_2)u\otimes d_{x,y}^+ +\langle w_1\,\vert\,
 w_2\rangle\hat{}{\,} \hat u'\otimes d_{x,y}^-+ (w_1\bullet w_2)\otimes
 xy]\\
 &=(w_1\,\vert\, w_2)\hat u'\otimes [d^-,d_{x,y}^+]+\langle w_1\,\vert\,
 w_2\rangle\hat{}{\,} u\otimes [d^-,d_{x,y}^-]\\
 &\hspace{140pt} +(\hat u'\diamond(w_1\bullet
 w_2))\otimes d^-(xy),\\[6pt]
[[\hat u'&\otimes d^-,w_1\otimes x],w_2\otimes y]\\
 &=[(\hat u'\diamond w_1)\otimes d^-(x),w_2\otimes y]\\
 &=(\hat u'\diamond w_1\,\vert\, w_2)u\otimes d_{d^-(x),y}^+ +\langle \hat u'\diamond w_1\,\vert\,
 w_2\rangle\hat{}{\,}\hat u'\otimes d_{d^-(x),y}^- \\
 &\hspace{140pt} +((\hat u'\diamond w_1)\bullet
 w_2)\otimes d^-(x)y\\[6pt]
[w_1&\otimes x,[\hat u'\otimes d^-,w_2\otimes y]]\\
 &=[w_1\otimes x,(\hat u'\diamond w_2)\otimes d^-(y)]\\
 &=(w_1\,\vert\, \hat u'\diamond w_2)u\otimes d_{x,d^-(y)}^++\langle w_1\,\vert\,
 \hat u'\diamond w_2\rangle\hat{}{\,} \hat u'\otimes d_{x,d^-(y)}^- \\
 &\hspace{140pt} +(w_1\bullet
 (\hat u'\diamond w_2))\otimes xd^-(y).
\end{split}
\]
But notice that $\langle w_1\,\vert\, w_2\rangle\hat{}{\,}=(\hat
u'\diamond w_1\,\vert\, w_2)=-(w_1\,\vert\, \hat u'\diamond w_2)$ and
$(w_1\,\vert\, w_2)=\langle \hat u'\diamond w_1\,\vert\,
w_2\rangle\hat{}{\,}=-\langle w_1\,\vert\, \hat u'\diamond
w_2\rangle\hat{}{\,}$. Also $\hat u'\diamond(w_1\bullet w_2)=-(\hat
u'\diamond w_1)\bullet w_2=-w_1\bullet(\hat u'\diamond w_2)$, so we
conclude:
\begin{equation}\label{eq:d-}
d^-(xy)=-d^-(x)y -xd^-(y)\quad\text{and}\quad
[d^-,d_{x,y}^\pm]=d_{d^-(x),y}^\mp -d_{x,d^-(y)}^\mp,
\end{equation}
for any $x,y\in M$ and $d^-\in\frd^+$.

Since $d_{x,y}^\pm=\mp d_{y,x}^\pm$ for any $x,y\in M$, equations
\eqref{eq:d+} and \eqref{eq:d-} show that, for any $a,b,x,y\in M$:
\[
\begin{split}
\{a,b,xy\}&=d_{a,b}^+(xy)+d_{a,b}^-(xy)\\
 &=d_{a,b}^+(x)y+xd_{a,b}^+(y)-d_{a,b}^-(x)y-xd_{a,b}^-(y)\\
 &=-\bigl(d_{b,a}^+(x)+d_{b,a}^-(x)\bigr)y-x\bigl(d_{b,a}^+(y)+d_{b,a}^-(y)\bigr)\\
 &=-\{b,a,x\}y-x\{b,a,y\},
\end{split}\]
thus obtaining \eqref{eq:idS3-1}, and
\[
\begin{split}
[d_{a,b}^+ +d_{a,b}^-,d_{x,y}^++d_{x,y}^-]
 &=d_{d_{a,b}^+(x),y}^++d_{x,d_{a,b}^+(y)}^+ +
   d_{d_{a,b}^+(x),y}^-+d_{x,d_{a,b}^+(y)}^- \\
   &\qquad +
   d_{d_{a,b}^-(x),y}^- -d_{x,d_{a,b}^-(y)}^- +
   d_{d_{a,b}^-(x),y}^+ -d_{x,d_{a,b}^-(y)}^+\\
 &=d_{\{a,b,x\},y}^+ + d_{\{a,b,x\},y}^- -
   d_{x,\{b,a,y\}}^+-d_{x,\{b,a,y\}}^-,
\end{split}
\]
and this gives \eqref{eq:idS3-2}.

The converse follows by straightforward computations of the Jacobi
identity.
\end{proof}

\smallskip

The reason behind the next definition will become clear in Theorem \ref{th:Malcev}.

\begin{definition}\label{de:generalizedMalcev}
A \emph{generalized Malcev algebra} is a triple $(M,xy,\{x,y,z\})$, where
$M$ is a vector space, $(x,y)\mapsto xy$ a bilinear anticommutative
product, and $(x,y,z)\mapsto \{x,y,z\}$ a trilinear product on $M$,
satisfying the conditions in \eqref{eq:identitiesS3}. The Lie
algebra $\frg(M)$ in \eqref{eq:gM}, with the action of $S_3$ in
Theorem \ref{th:S3}, is called the \emph{associated Lie algebra}.
\end{definition}

Therefore, Theorem \ref{th:S3} asserts that there is a generalized Malcev algebra attached to any Lie algebra with $S_3$-action, assuming that the cycle $\varphi$ acts nontrivially.

\begin{remark}\label{re:omegaink}
Let $\frg$ be a Lie algebra endowed with an action of $S_3$ by
automorphisms, so that $\varphi$ acts nontrivially, as in Theorem \ref{th:S3}, and assume that $\omega\in k$, where $\omega^3=1\ne \omega$. Then $\frg$ decomposes as $\frg=\frg_1\oplus\frg_\omega\oplus\frg_{\omega^2}$, where $\frg_\mu=\{ x\in\frg : \varphi(x)=\mu x\}$ for $\mu=1,\omega,\omega^2$. Then in the decomposition in \eqref{eq:gd+d-WM} we have:
\[
\frg_1=(U\otimes \frd^+)\oplus(U'\otimes\frd^-),\quad \frg_\omega\oplus\frg_{\omega^2}=W\otimes M.
\]
Besides, $\varphi(w_\omega)=\omega w_{\omega}$ and $\varphi(w_{\omega^2})=\omega^2w_{\omega^2}$, so $\frg_\omega=w_\omega\otimes M$ and $\frg_{\omega^2}=w_{\omega^2}\otimes M$. Also, for any $x,y,z\in M$:
\[
\begin{split}
[\tau(w_\omega\otimes x),\tau(w_{\omega}\otimes y)]&=
 [w_{\omega^2}\otimes x,w_{\omega^2}\otimes y]=w_{\omega}\otimes xy,\\
[[w_\omega\otimes x,\tau(w_\omega\otimes y)],w_\omega\otimes z]&=
 [[w_\omega\otimes x,w_{\omega^2}\otimes y],w_\omega\otimes z]\\
   &=[u\otimes d_{x,y}^++\hat u'\otimes d_{x,y}^-,w_\omega\otimes z]\\
   &=w_\omega\otimes\{x,y,z\},
\end{split}
\]
because of \eqref{eq:uprimewomega}, \eqref{eq:uhatprimewomega} and \eqref{eq:bracketS3bis}. Therefore $M$ can be identified to $\frg_\omega$ ($x\leftrightarrow w_\omega\otimes x$) with:
\[
\begin{cases} xy=[\tau(x),\tau(y)],\\
\{x,y,z\}=[[x,\tau(y)],z],\end{cases}
\]
for any $x,y,z\in\frg_\omega$.

Conversely, given a generalized Malcev algebra $(M,xy,\{x,y,z\})$, the associated Lie algebra $\frg(M)$ can be identified to
\[
\frd^+\oplus\frd^-\oplus \nu_0(M)\oplus\nu_1(M),
\]
where $\nu_i(M)$ denotes a copy of $M$ (think of $\nu_0(x)$ as $w_\omega\otimes x$ and $\nu_1(x)$ as $w_{\omega^2}\otimes x$, and identify $u\otimes d^+$ to $d^+$ and $\hat u'\otimes d^-$ to $d^-$) with Lie bracket given by:
\begin{itemize}
\item $[d_1^\pm,d_2^\pm]$ is given by the usual Lie bracket in $\frgl(M)$,
\item $[d^+,\nu_i(x)]=\nu_i(d^+(x))$,
\item $[d^-,\nu_i(x)]=(-1)^i\nu_i(d^-(x))$,
\item $[\nu_i(x),\nu_i(y)]=\nu_{1-i}(xy)$,
\item $[\nu_0(x),\nu_1(y)]=d_{x,y}^++d_{x,y}^-$,
\end{itemize}
for any $d_1^\pm,d_2^\pm,d^\pm\in\frd^\pm$, $x\in M$, and $i=0,1$.\hfill\qed
\end{remark}

It can be shown that some of the conditions \eqref{eq:idS3-3} in Definition \ref{de:generalizedMalcev} are redundant:

\begin{proposition}\label{pr:redundant}
Let $(M,xy,\{x,y,z\})$ be a triple where
$M$ is a vector space, $(x,y)\mapsto xy$ a bilinear anticommutative
product, and $(x,y,z)\mapsto \{x,y,z\}$ a trilinear product on $M$,
satisfying the conditions \eqref{eq:idS3-4}, \eqref{eq:idS3-1} and \eqref{eq:idS3-2}. Then:
\begin{romanenumerate}
\item For any $x,y,z,t\in M$, the equation
\[
\{xy,z,t\}+\{yz,x,t\}+\{zx,y,t\}=0
\]
holds. (This is the first part of \eqref{eq:idS3-3}.)
\item If either $M=M^2$ or $\{x\in M: xM^2=0\}=0$ hold, then for any $x,y,z,t\in M$,
\[
\{x,yz,t\}+\{y,zx,t\}+\{z,xy,t\}=0
\]
holds too. (This is the second part of \eqref{eq:idS3-3}.)
\end{romanenumerate}
\end{proposition}
\begin{proof}
For any $x,y,z,t\in M$, equations \eqref{eq:idS3-4} and \eqref{eq:idS3-1}, together with the anticommutativity of the bilinear product give:
\[
\begin{split}
((xy)t)z&=\{xy,z,t\}-\{t,z,xy\}\\
 &=\{xy,z,t\}+\{z,t,x\}y-\{z,t,y\}x.
\end{split}
\]
Therefore,
\[
\begin{split}
((xy)t)z+((yz)t)x&+((zx)t)y\\
 &=\{xy,z,t\}+\{yz,x,t\}+\{zx,y,t\} \\
 &\qquad +\{z,t,x\}y-\{z,t,y\}x + \{x,t,y\}z\\
 &\qquad\quad -\{x,t,z\}y +
   \{y,t,z\}x-\{y,t,x\}z\\
 &=\{xy,z,t\}+\{yz,x,t\}+\{zx,y,t\} \\
 &\qquad +\bigl(\{y,t,z\}-\{z,t,y\}\bigr)x +
   \bigl(\{z,t,x\}-\{x,t,z\}\bigr)y\\
   &\qquad\quad +
   \bigl(\{x,t,y\}-\{y,t,x\}\bigr)z\\
 &=\{xy,z,t\}+\{yz,x,t\}+\{zx,y,t\} \\
 &\qquad + ((xy)t)z+((yz)t)x+((zx)t)y,
\end{split}
\]
by \eqref{eq:idS3-4}, and this proves item (i).

Now, from equation \eqref{eq:idS3-2}  we have
\[
\{xy,z,\{a,t,b\}\}-\{a,t,\{xy,z,b\}\}
 =\{\{xy,z,a\},t,b\}-\{a,\{z,xy,t\},b\},
\]
which, together with (i) gives
\[
\{a,\{x,yz,t\}+\{y,zx,t\}+\{z,xy,t\},b\}=0,
\]
for any $a,b,x,y,z,t\in M$, which by \eqref{eq:idS3-4} yields
\[
M^2\bigl(\{x,yz,t\}+\{y,zx,t\}+\{z,xy,t\}\bigl)=0.
\]
Also, equation \eqref{eq:idS3-1}, together with (i), gives
\[
\{x,yz,u\}+\{y,zx,u\}+\{z,xy,u\}=0
\]
for any $x,y,z,t\in M$ and $u\in M^2$. Hence item (ii) follows.
\end{proof}

\begin{proposition}\label{pr:gsimpleS3}
Let $\frg$ be a Lie algebra endowed with an action of $S_3$ by
automorphisms and assume that $\varphi$ acts nontrivially. Let
$(M,xy,\{x,y,z\})$ be the corresponding generalized Malcev algebra,
and let $\frg(M)$ be the associated Lie algebra. If $\frg$ does not
contain proper ideals invariant under the action of $S_3$ (in
particular, if $\frg$ is simple), then $\frg$ and $\frg(M)$ are
isomorphic as Lie algebras with $S_3$-action.
\end{proposition}
\begin{proof}
Write $\frg=(U\otimes \frd^+)\oplus(U'\otimes \frd^-)\oplus(W\otimes
M)$ as in \eqref{eq:gd+d-WM}. Then
\[
\bigl(U\otimes\{d\in \frd^+:d(M)=0\}\bigr)\oplus
\bigl(U'\otimes\{d\in\frd^-: d(M)=0\}\bigr)
\]
is an ideal of $\frg$ invariant under the action of $S_3$, and
hence $\{d\in \frd^\pm: d(M)=0\}=0$, so it can be assumed that
$\frd^+,\frd^-$ are contained in $\frgl(M)$. Besides,
\[
 (U\otimes
d_{M,M}^+)\oplus (U'\otimes d_{M,M}^-)\oplus (W\otimes M)
\]
is the ideal of $\frg$ generated by $M$, which is necessarily the
whole $\frg$. Thus, $\frd^\pm=d_{M,M}^\pm$ and the result follows.
\end{proof}

\begin{example}\label{ex:G2} \textbf{(Simple Lie algebra of type $G_2$)}\newline
Assume $\omega\in k$, and let $E$ be a three dimensional vector space over $k$ endowed with a nondegenerate symmetric bilinear form $b$ with trivial determinant (that is $\det\bigl(b(e_i,e_j)\bigr)$ is a square for any basis $\{e_1,e_2,e_3\}$ of $E$). Fix then an orthogonal basis $\{e_1,e_2,e_3\}$ of $E$ with $\det\bigl(b(e_i,e_j)\bigr)=1$ and consider the trilinear alternating map $T:E\times E\times E\rightarrow k$ determined by $T(e_1,e_2,e_3)=1$. Then for any $u_1,u_2,u_3\in E$, $T(u_1,u_2,u_3)^2=\det\bigl(b(u_i,u_j)\bigr)$. There appears the attached cross product $v\times w$ determined by $b(u,v\times w)=T(u,v,w)$ for any $u,v,w\in E$, which satisfies the equation
\[
(u\times v)\times w=b(u,w)v-b(v,w)v
\]
for any $u,v,w\in E$ or, equivalently,
\[
b(x\times y,z\times t)=\begin{vmatrix} b(x,z)&b(x,t)\\ b(y,z)&b(y,t)\end{vmatrix}
\]
for any $x,y,z,t\in E$.

The split Lie algebra of type $G_2$ can be described (see \cite[page 348]{FH}) as
\[
\frg=E^*\oplus\frsl(E)\oplus E
\]
($E^*$ denotes the dual vector space), with bracket given by the usual bracket in the Lie subalgebra of zero trace endomorphisms $\frsl(E)$ and by:
\[
\left\{
\begin{aligned}
&[A,e]=A(e),\quad [A,f]=-fA\ \text{(composition of maps)},\\
&[e_1,e_2]=-2e_1\wedge e_2\in E^*,\ \text{where $e_1\wedge e_2: e\mapsto T(e_1,e_2,e)$,}\\
&[f_1,f_2]=2f_1\wedge f_2\in E,\\
 &\qquad\qquad\qquad\text{such that $T(f_1\wedge f_2,e_1,e_2)=f_1(e_1)f_2(e_2)-f_1(e_2)f_2(e_1)$,}\\
&[e,f]=3f(.)e-f(e)I,\ \text{$I$ denotes the identity map on $E$,}
\end{aligned}\right.
\]
for any $e,e_1,e_2\in E$, $f,f_1,f_2\in E^*$ and $A\in \frsl(E)$.

The symmetric group $S_3$ acts on $\frg$ by automorphisms as follows:
\[
\begin{split}
&\varphi(e)=\omega e,\ \varphi(f)=\omega^2f,\ \varphi(A)=A,\\
&\tau(e)=-b(e,.),\ \tau(A)=-A^t,
\end{split}
\]
for any $e\in E$, $f\in F$ and $A\in \frsl(E)$, where $A^t$ denotes the adjoint relative to $b$, that is, $b(Ae,e')=b(e,A^te')$ for any $e,e'\in E$.

Then $\frg_\omega=E$ is a generalized Malcev algebra (see Remark \ref{re:omegaink}) with
\[
\begin{split}
&xy=[\tau(x),\tau(y)]=2b(x,.)\wedge b(y,.)=2x\times y,\\
&\{x,y,z\}=[[x,\tau(y)],z]=-[3b(y,.)x-b(x,y)I,z]=b(x,y)z-3b(y,z)x,
\end{split}
\]
for any $x,y,z\in E$.

By Proposition \ref{pr:gsimpleS3}, $\frg$ is isomorphic to the associated Lie algebra $\frg(E)$. Thus, the central simple split Lie algebra of type $G_2$ is
completely determined by the three dimensional generalized Malcev algebra $(E,xy,\{x,y,z\})$.\hfill\qed
\end{example}

\bigskip
\section{Lie algebras with $S_3$-action and Malcev
algebras}\label{se:S3Malcev}

Some noteworthy examples of generalized Malcev algebras are provided
by Malcev algebras (see \cite{Sagle}).

\begin{theorem}\label{th:Malcev}
Let $(M,xy,\{x,y,z\})$ be a generalized Malcev algebra such that
$\{x,y,z\}$ is skew-symmetric on $x$ and $y$. Then $(M,xy)$ is a
Malcev algebra and
\begin{equation}\label{eq:triplebinary}
2\{x,y,z\}=(xy)z+x(yz)-y(xz)
\end{equation}
for any $x,y,z\in M$.

Conversely, if $(M,xy)$ is a Malcev algebra and a triple product
$\{.,.,.\}$ is defined on $M$ by the formula in
\eqref{eq:triplebinary}, then $(M,xy,\{x,y,z\})$ is a
generalized Malcev algebra.
\end{theorem}
\begin{proof}
This result appears essentially, with a very different formulation,
in \cite[Theorem 1]{Mikheev} and in \cite{Grishkov}. We include a
proof for completeness.

Assume first that $(M,xy,\{x,y,z\})$ is a generalized Malcev algebra
such that $\{x,y,z\}$ is skew-symmetric on $x$ and $y$. Then
\eqref{eq:idS3-4} gives
\[
\begin{split}
(xy)z&+x(yz)-y(xz)\\
 &=(xy)z-(yz)x+(xz)y\\
 &=\{x,z,y\}-\{y,z,x\}-\{y,x,z\}+\{z,x,y\}+\{x,y,z\}-\{z,y,x\}\\
 &=2\{x,y,z\},
\end{split}
\]
so \eqref{eq:triplebinary} holds.

Now, with $J(x,y,z)=(xy)z+(yz)x+(zx)y$, this is equivalent to
\begin{equation}\label{eq:Jbrackets}
2(xy)z-J(x,y,z)=2\{x,y,z\},
\end{equation}
for any $x,y,z\in M$, and then
\[
\begin{split}
J(x,y,xy)&=2(xy)(xy)-2\{x,y,xy\}=-2\{x,y,xy\}\quad\text{(as $z^2=0$
for any $z$)}\\
&=2\{y,x,x\}y+2x\{y,x,y\}\quad\text{(by \eqref{eq:idS3-1})}\\
&=2((yx)x)y+2x((yx)y)\quad\text{(by \eqref{eq:triplebinary})}\\
&=-2J(x,y,xy).
\end{split}
\]
Hence $J(x,y,xy)=0$ for any $x,y\in M$, so
\begin{equation}\label{eq:JJ}
J(x,y,xz)+J(x,z,xy)=0
\end{equation}
for any $x,y,z\in M$. Now equation \eqref{eq:idS3-3} gives
\[
\{xy,z,x\}+\{yz,x,x\}+\{zx,y,x\}=0,
\]
which, by \eqref{eq:Jbrackets} becomes
\[
\begin{split}
0&=J(x,y,z)x-\frac{1}{2}\bigl(J(xy,z,x)+J(yz,x,x)+J(zx,y,x)\bigr)\\
&=J(x,y,z)x-J(x,y,xz),
\end{split}
\]
because of \eqref{eq:JJ}. Hence $J(x,y,xz)=J(x,y,z)x$ for any
$x,y,z\in M$, and this is equivalent to $M$ being a Malcev algebra
(see \cite{Sagle}).

Conversely, if $(M,xy)$ is a Malcev algebra and the triple product
is defined by \eqref{eq:triplebinary}, then for any $x,y$:
\[
\{x,y,.\}=\frac{1}{2}\bigl(\ad_{xy}+[\ad_x,\ad_y]\bigl)=D(x,y),
\]
($\ad_x:y\mapsto xy$), which is known to be a derivation of $(M,xy)$
 satisfying $D(xy,z)+D(yz,x)+D(zx,y)=0$ (see \cite{Sagle}).
Hence the conditions \eqref{eq:idS3-1}, \eqref{eq:idS3-2} and
\eqref{eq:idS3-3} are satisfied. Also,
$2\bigl(\{x,z,y\}-\{y,z,x\}\bigr)=(xz)y+x(zy)-z(xy)-(yz)x-y(zx)+z(yx)=2(xy)z$
for any $x,y,z\in M$ because of the anticommutativity of the
product, so condition \eqref{eq:idS3-4} holds too.
\end{proof}

\smallskip

Note that in Example \ref{ex:G2} there appears a three-dimensional generalized Malcev algebra $(E,xy,\{x,y,z\})$, where $(E,xy)$ is a Lie (and hence Malcev) algebra, but where the triple product $\{x,y,z\}$ is not given by formula \eqref{eq:triplebinary}, as it is not even skew symmetric in $x$ and $y$.

\smallskip

\begin{remark}\label{re:Lietriality}
Let $(M,xy)$ be a Malcev algebra, and define the triple product
$\{.,.,.\}$ by \eqref{eq:triplebinary}. Then $\{x,y,.\}+\{y,x,.\}$
is identically $0$, so the associated Lie algebra $\frg(M)$ is just
\[
\frg(M)=(U\otimes \frd^+)\oplus (W\otimes M),
\]
which is a \emph{Lie algebra with triality} in the notation of
\cite{Grishkov}. This is a Lie algebra with an action of $S_3$ by
automorphisms such that the alternating module does not appear in
the decomposition into a direct sum of irreducible $S_3$-modules.

Besides, if $(M,xy)$ is a Lie algebra, then $\{x,y,.\}=\frac{1}{2}\bigl(\ad_{xy}+[\ad_x,\ad_y]\bigr)=\ad_{xy}$, so $\frg(M)=(U\otimes \ad_M)\oplus (W\otimes M)$. If the center of $M$ is trivial, then $\ad_M$ can be identified with $M$ itself, and hence $\frg(M)\simeq(U\oplus W)\otimes M=k^3\otimes M\simeq M^3$. In this way, $\frg(M)$ is seen to be isomorphic to the direct sum of three copies of $M$, with componentwise multiplication.
\hfill\qed
\end{remark}

\smallskip

\begin{remark}
With the same conventions as in Remark \ref{re:Lietriality}, the transposition $\tau$ provides a
$\bZ_2$-grading of $\frg(M)$ whose even part is $(u\otimes
\frd^+)\oplus (w_+\otimes M)$. Since $w_+\bullet w_+=w_+$, $(w_+\,\vert\,
w_+)=2$ and $d_{x,y}^+=\{x,y,.\}=\frac{1}{2}D(x,y)$, it follows that
this even part is isomorphic to
\[
D(M,M)\oplus M,
\]
with product given by the Lie bracket of derivations in $D(M,M)$,
and by $[d,x]=d(x)$ and $[x,y]=D(x,y)+xy$ for any $d\in D(M,M)$ and
$x,y\in M$. If the center of $M$ is trivial, $M$ can be identified
to $\ad_M$. Then the Lie multiplication algebra of $M$ is
$L(M)=D(M,M)\oplus\ad_M$ (see \cite{Sagle}), and the map
$D(M,M)\oplus M\rightarrow L(M)$, $D(x,y)\mapsto D(x,y)$, $x\mapsto
-\ad_x$ is an isomorphism. That is, the fixed subalgebra of
$\frg(M)$ by the action of $\tau$ is isomorphic to the Lie
multiplication algebra of $M$.\hfill\qed
\end{remark}

\begin{remark}
Given a Malcev algebra $M$, in \cite{ChemaIvan} a Lie algebra
$\calL(M)$ is defined by generators $\{\lambda_a,\rho_a:a\in M\}$
and relations (see \cite[Eq. (1)]{ChemaIvan}):
\[
\begin{split}
&\lambda_{\alpha a+\beta b}=\alpha\lambda_a+\beta\lambda_b,\quad
 \rho_{\alpha a+\beta b}=\alpha\rho_a+\beta\rho_b,\\
&[\lambda_a,\lambda_b]=\lambda_{ab}-2[\lambda_a,\rho_b],\quad
 [\rho_a,\rho_b]=-\rho_{ab}-2[\lambda_a,\rho_b],\\
&[\lambda_a,\rho_b]=[\rho_a,\lambda_b],
\end{split}
\]
for any $a,b\in M$ and $\alpha,\beta\in k$. Then $S_3$ acts on this
Lie algebra by automorphisms as follows:
\[
\begin{split}
&\tau(\lambda_a)=-\rho_a\\
&\varphi(\lambda_a)=-\lambda_a-\rho_a,\ \varphi(\rho_a)=\lambda_a.
\end{split}
\]
It is straightforward to check that the relations are preserved.
Actually, $\tau$ coincides with the automorphism $\zeta\eta\zeta$
while $\varphi$ coincides with the automorphism $\zeta\eta$ in
\cite[p.~386]{ChemaIvan}. The arguments in the proof of Proposition
3.2 in \cite{ChemaIvan} can be used to prove that the envelope
$\calL(M)$ is isomorphic to our $\frg(M)$, by means of
\[
\begin{split}
\rho_a-\lambda_a&\mapsto w_+\otimes a,\\
\rho_a+\lambda_a&\mapsto \frac{1}{3}w_-\otimes a.\ \qed
\end{split}
\]
\end{remark}

\medskip

Recall that given any nonassociative algebra $(A,\cdot)$, the
generalized alternative nucleus is the subspace
\[
N_{alt}(A,\cdot)=\{ a\in A: (a,x,y)=-(x,a,y)=(x,y,a)\ \forall x,y\in
A\},
\]
where $(x,y,z)=(x\cdot y)\cdot z-x\cdot (y\cdot z)$ is the
associator of the elements $x,y,z$. The generalized alternative
nucleus is a Malcev algebra under the commutator $[x,y]=x\cdot
y-y\cdot x$ (see \cite{ChemaIvan}).

Given a unital algebra with involution $(A,\cdot,-)$, it has been
shown in Corollary \ref{co:S3lrta} that $\lrt(A,\cdot,-)$ is a Lie
algebra with $S_3$-action ($S_3\simeq S_4/V_4$), with no irreducible
component isomorphic to the alternating module.

\begin{proposition}
Let $(A,\cdot,-)$ be a unital algebra with involution. Then the
generalized Malcev algebra associated to the Lie algebra
$\lrt(A,\cdot,-)$ is isomorphic to $(M,[x,y],\{x,y,z\})$, where
\[
M=\{a\in N_{alt}(A,\cdot): \bar a=-a\},
\]
and where $2\{x,y,z\}=[[x,y],z]+[x,[y,z]]-[y,[x,z]]$ for any
$x,y,z\in M$.
\end{proposition}
\begin{proof}
The result in \cite[Corollary 3.5]{AllisonFaulkner} shows that
\[
\begin{split}
\lrt(&A,\cdot,-)\\
 &=\{(d,d,d): d\in \der(A,\cdot,-)\}\\
&\quad \oplus
\{(L_{s_2}-R_{s_3},L_{s_3}-R_{s_1},L_{s_1}-R_{s_2}):s_1,s_2,s_3\in
M,\, s_1+s_2+s_3=0\},
\end{split}
\]
with $M$ as above. Here $L$ and $R$ denote the left and right
multiplications in $(A,\cdot)$. This shows that, as vector spaces,
there is an isomorphism:
\[
\begin{split}
\Psi:(U\otimes \der(A,\cdot,-))\oplus(W\otimes M)&\rightarrow
\lrt(A,\cdot,-)\\
u\otimes d\qquad&\mapsto \ (d,d,d)\\
(\alpha_1,\alpha_2,\alpha_3)\otimes s\qquad&\mapsto
\bigl(L_{\alpha_2s}-R_{\alpha_3s},
L_{\alpha_3s}-R_{\alpha_1s},L_{\alpha_1s}-R_{\alpha_2s}\bigr)\\[-4pt]
\text{\small ($\alpha_1+\alpha_2+\alpha_3=0$)}\qquad &
\end{split}
\]
for $d\in\der(A,\cdot,-)$, $\alpha_i\in k$, $i=1,2,3$, and $s\in M$.

Now, given $w_1=(\alpha_1,\alpha_2,\alpha_3),\
w_2=(\beta_1,\beta_2,\beta_3)\in W$ (so
$\alpha_1+\alpha_2+\alpha_3=0=\beta_1+\beta_2+\beta_3$), and $s,t\in
M$, the arguments in the proof of \cite[Corollary
3.5]{AllisonFaulkner} give:
{\small
\[
\begin{split}
\bigl[\bigl(L_{\alpha_2s}-R_{\alpha_3s},
L_{\alpha_3s}-R_{\alpha_1s},&L_{\alpha_1s}-R_{\alpha_2s}\bigr),
\bigl(L_{\beta_2t}-R_{\beta_3t},
L_{\beta_3t}-R_{\beta_1t},L_{\beta_1t}-R_{\beta_2t}\bigr)\bigr]\\
&=\frac{1}{3}(D,D,D)+\bigl(L_{a_2}-R_{a_3},L_{a_3}-R_{a_1},L_{a_1}-R_{a_2}\bigr)
\end{split}
\]}
where:
\[
\begin{split}
D&=[L_{\alpha_2s}-R_{\alpha_3s},L_{\beta_2s}-R_{\beta_3t}]\\
&\qquad +
[L_{\alpha_3s}-R_{\alpha_1s},L_{\beta_3s}-R_{\beta_1t}] +
[L_{\alpha_1s}-R_{\alpha_2s},L_{\beta_1s}-R_{\beta_2t}]\\
 & =
(\alpha_1\beta_1+\alpha_2\beta_2+\alpha_3\beta_3)\bigl([L_s,L_t]+[R_s,R_t]\bigr)\\
&\qquad -(\alpha_2\beta_3+\alpha_3\beta_1+\alpha_1\beta_2)[L_s,R_t]
 -
(\alpha_3\beta_2+\alpha_1\beta_3+\alpha_2\beta_1)[R_s,L_t].
\end{split}
\]
But
\[
[L_s,R_t](x)=-(s,x,t)=(t,x,s)=-[L_t,R_s](x),
\]
so
$[L_s,R_t]=[R_s,L_t]$. Besides,
\[
\sum_{1\leq i\ne j\leq
3}\alpha_i\beta_j=(\alpha_1+\alpha_2+\alpha_3)(\beta_1+\beta_2+\beta_3)
-\sum_{i=1}^3\alpha_i\beta_i.
\]
Hence:
\[
\begin{split}
D&=(\alpha_1\beta_1+\alpha_2\beta_2+\alpha_3\beta_3)
  \bigl([L_s,L_t]+[R_s,R_t]+[L_s,R_t]\bigr)\\
  &=3(w_1\,\vert\, w_2)\bigl([L_s,L_t]+[R_s,R_t]+[L_s,R_t]\bigr).
\end{split}
\]

Also, $a_1=\frac{1}{3}(b_3-b_2)$ (and cyclically), with
$b_1=[L_{\alpha_2s}-R_{\alpha_3s},L_{\beta_2t}-R_{\beta_3t}](1)$
(and cyclically). Hence
\[
\begin{split}
b_1&=(\beta_2-\beta_3)(\alpha_2s\cdot t-\alpha_3t\cdot s)
  -(\alpha_2-\alpha_3)(\beta_2t\cdot s-\beta_3s\cdot t)\\
  &=(\alpha_2\beta_2-\alpha_3\beta_3)[s,t],
\end{split}
\]
and
\[
a_1=\frac{1}{3}(2\alpha_1\beta_1-\alpha_2\beta_2-\alpha_3\beta_3)[s,t].
\]

Since
\[
w_1\bullet
w_2=\frac{1}{3}\bigl(2\alpha_1\beta_1-\alpha_2\beta_2-\alpha_3\beta_3,
2\alpha_2\beta_2-\alpha_3\beta_3-\alpha_1\beta_1,
2\alpha_3\beta_3-\alpha_1\beta_1-\alpha_2\beta_2\bigr),
\]
it follows that
\[
\begin{split}
[\Psi(&w_1\bullet s),\Psi(w_2\bullet t)]\\
&=
 \Psi\bigl((w_1\,\vert\, w_2)u\otimes
 \bigl([L_s,L_t]+[R_s,R_t]+[L_s,R_t]\bigr)\bigr)+
 \Psi\bigl((w_1\bullet w_2)\otimes [s,t]\bigr)
\end{split}
\]
and this shows that the attached generalized Malcev algebra is
the triple $(M,[x,y],\{x,y,z\})$, where
$\{x,y,z\}=\bigl([L_x,L_y]+[R_x,R_y]+[L_x,R_y]\bigr)(z)$. Since
$(M,[.,.])$ is a Malcev algebra (see \cite{ChemaIvan}), Theorem \ref{th:Malcev} finishes
the proof.
\end{proof}

\medskip

A distinguished example is given by any octonion algebra
$(\bO,\cdot,-)$ over $k$ with its standard involution. Since $\bO$
is alternative, $N_{alt}(\bO,\cdot)=\bO$, so $M=\bO_0=\{x\in\bO:
t(x)=x+\bar x=0\}$. Then, it is well-known by the Principle of
Triality (see \cite[Chapter VIII]{KMRT} and references therein) that
$\lrt(\bO,\cdot,-)$ is isomorphic to the orthogonal Lie algebra
$\frso(\bO,n)$, where $n$ denotes the norm of $\bO$
($n(x)=x\cdot\bar x$). The action of $S_3$ decomposes this Lie
algebra as
\[
\frso(\bO,n)=(U\otimes \der(\bO,\cdot))\oplus(W\otimes \bO_0).
\]
The associated Malcev algebra is the central simple non-Lie Malcev algebra $\bO_0$. Recall that any central simple non-Lie Malcev algebra appears in this way (see \cite{Fil76} and \cite{Kuz68}).

\bigskip
\section{Examples}\label{se:S3examples}

Other noteworthy examples of generalized Malcev algebras consist of
those whose bilinear multiplication is trivial. Equations
\eqref{eq:identitiesS3} show that these are just the Jordan triple
systems (see \cite{Meyberg}).

Given a Jordan triple system $T$ with triple product denoted by
$\{x,y,z\}$, consider the Lie subalgebra of $\frgl(T)\oplus\frgl(T)$
defined by:
\begin{multline*}
\{(d_1,d_2): d_i(\{x,y,z\})=\{d_i(x),y,z\}
 +\{x,d_{3-i}(y),z\}+
 \{x,y,d_i(z)\}\\ \forall x,y,z\in T,\, \forall i=1,2\}.
\end{multline*}
Then $\frs(T)=\espan{\bigl(\{x,y,.\},-\{y,x,.\}\bigr): x,y\in
T}$ is a subalgebra of $\frgl(T)\oplus\frgl(T)$ and the
Tits-Kantor-Koecher Lie algebra of $T$ is the Lie algebra
\[
\calK(T)=TKK(T)=T\oplus\frs(T)\oplus \hat T,
\]
where $\hat T$ is just a copy of $T$ (and given an element $x\in T$,
its copy in $\hat T$ will be denoted by $\hat x$), with bracket
given by:
\[
\begin{split}
&[T,T]=[\hat T,\hat T]=0,\ \text{$\frs(T)$ is a subalgebra of
$\calK(T)$,}\\
&[x,\hat y]=\bigl(\{x,y,.\},-\{y,x,.\}\bigr),\\
&[(d_1,d_2),x]=d_1(x),\ [(d_1,d_2),\hat y]=\widehat{d_2(y)},
\end{split}
\]
for any $(d_1,d_2)\in\frs(T)$ and $x,y\in T$.

\begin{proposition}
Assume $\omega\in k$ ($\omega^3=1\ne \omega$), and let
$(T,\{x,y,z\})$ be a Jordan triple system. Then $(T,xy,\{x,y,z\})$
is a generalized Malcev algebra with $xy= 0$ for any $x,y\in T$,
and the linear map
\[
\begin{split}
\frg(T)&\rightarrow \calK(T)\\
w_{\omega}\otimes x&\mapsto x\\
w_{\omega^2}\otimes x&\mapsto \hat x\\
u\otimes d^+&\mapsto (d^+,d^+)\\
\hat u'\otimes d^-&\mapsto (d^-,-d^-)
\end{split}
\]
for $x\in T$,
$d^+\in\espan{d_{x,y}^+=\frac{1}{2}\bigl(\{x,y,.\}-\{y,x,.\}\bigr):
x,y\in T}$ and
$d^-\in\espan{d_{x,y}^-=\frac{1}{2}\bigl(\{x,y,.\}+\{y,x,.\}\bigr):x,y\in
T}$, is an isomorphism of Lie algebras.
\end{proposition}
\begin{proof}
The proof is straightforward from the different definitions
involved.
\end{proof}

\begin{example}\label{ex:so} \textbf{(Orthogonal Lie algebras)}\newline
Let us assume, for simplicity, that the ground field $k$ is algebraically closed, and let $\calV$ be a vector space of dimension $\geq 3$, endowed with a nondegenerate symmetric bilinear form $b$. Let $\{e_1,\ldots,e_n\}$ be a basis with $b(e_i,e_j)=\delta_{i,j}$. The symmetric group $S_3$ acts naturally by isometries by permuting the first three vectors in this basis and leaving the remaining ones fixed. Then $\calV$ decomposes, as a module for $S_3$, into the orthogonal sum of a submodule isomorphic to the two dimensional irreducible module $W$ and a direct sum of $n-2$ trivial modules. Thus, we may think of $\calV$ as the orthogonal sum $W\perp E$ with the restriction of $b$ to $W$ given by the bilinear form $(.\,\vert\, .)$ in \eqref{eq:WWWU} (by invariance under the action of $S_3$), and where $S_3$ acts trivially on $E$.

The orthogonal Lie algebra $\frg=\frso(\calV,b)$ is spanned by the maps:
\[
\sigma_{u,v}:w\mapsto b(u,w)v-b(v,w)u
\]
for any $u,v,w\in \calV$. Note that these maps satisfy
\[
[\sigma_{u,v},\sigma_{x,y}]
 =\sigma_{\sigma_{u,v}(x),y}+\sigma_{x,\sigma_{u,v}(y)}
\]
for any $u,v,x,y$.

The action of the symmetric group $S_3$ on $\calV$ induces an action on $\frso(\calV,b)$, where for any $\gamma\in S_3$ and $f\in \frso(\calV,b)$, $\gamma\cdot f=\gamma f\gamma^{-1}$. In this way, $\gamma\cdot \sigma_{u,v}=\sigma_{\gamma(u),\gamma(v)}$ for any $u,v\in \calV$. Then there is the natural decomposition $\frg=\frso(W)\oplus\frso(E)\oplus \sigma_{W,E}\simeq \frso(W)\oplus\frso(E)\oplus W\otimes E$, where $\frso(W)$ acts trivially on $E$ and $\frso(E)$ acts trivially on $W$. Besides, the one dimensional Lie algebra $\frso(W)$ is spanned by $\sigma_{w_\omega,w_{\omega^2}}$.

The action of the cycle $\varphi$ decomposes $\frg$ into the direct sum of its eigenspaces, where $\frg_1=\frso(W)\oplus\frso(E)$, $\frg_\omega=\sigma_{w_\omega,E}$ and $\frg_{\omega^2}=\sigma_{w_{\omega^2},E}$. (Recall that $\varphi(w_\omega)=\omega w_\omega$ and $\varphi(w_{\omega^2})=\omega^2w_{\omega^2}$) The attached generalized Malcev algebra is defined then on $\frg_\omega$ by Remark \ref{re:omegaink}. Note that for any $x,y,z\in E$:
\[
\begin{split}
[\tau(\sigma_{w_\omega,x}),\tau(\sigma_{w_\omega,y})]&=
  [\sigma_{w_{\omega^2},x},\sigma_{w_{\omega^2},y}]\\
  &=\sigma_{\sigma_{w_{\omega^2},x}(w_{\omega^2}),y}
    +\sigma_{w_{\omega^2},\sigma_{w_{\omega^2},x}(y)}=0,\\
[[\sigma_{w_\omega,x},\tau(\sigma_{w_\omega,y})],\sigma_{w_\omega,z}]&=
   [[\sigma_{w_\omega,x},\sigma_{w_{\omega^2},y}],\sigma_{w_\omega,z}]\\
   &=[\sigma_{x,y}+b(x,y)\sigma_{w_\omega,w_{\omega^2}},\sigma_{w_\omega,z}]\\
   &=\sigma_{w_\omega,\sigma_{x,y}(z)-b(x,y)z}\\
   &=\sigma_{w_\omega,b(x,z)y-b(y,z)x-b(x,y)z},
\end{split}
\]
which shows that the attached generalized Malcev algebra can be identified with $(E,xy,\{x,y,z\})$ with $xy=0$ and $\{x,y,z\}=b(x,z)y-b(y,z)x-b(x,y)z$ for any $x,y,z\in E$. This is the Jordan triple system associated to the nondegenerate symmetric bilinear form $-b\vert_E$ (see \cite{Meyberg}).\hfill\qed
\end{example}

\smallskip

Another easy example, which is neither a Malcev algebra nor a Jordan triple system is obtained with similar arguments, but with $\frso(\calV)$ substituted by $\frsl(\calV)$:

\begin{example}\label{ex:sl} \textbf{(Special Lie algebras)}\newline
Let $\calV$ be a vector space of dimension $\geq 3$ over an algebraically closed field $k$, and take a basis $\{e_1,\ldots, e_n\}$. Then $S_3$ acts by permuting the first three elements of this basis. As in Example \ref{ex:so}, this decomposes $\calV$ into a direct sum $\calV=W\oplus E$, where $W$ is the two dimensional irreducible module for $S_3$ and $E$ is a vector space on which $S_3$ acts trivially. Identify the general linear Lie algebra $\frgl(\calV)$ with $\calV\otimes \calV^*$ ($v\otimes f\in \calV\otimes \calV^*\leftrightarrow f(.)v\in\frgl(\calV)$, where $f(.)v$ is the linear map $w\mapsto f(w)v$ for any $v,w\in \calV$ and $f\in\calV^*$). Recall that $W$ is endowed with a unique (up to scalars) nondegenerate symmetric bilinear form $(.\,\vert\,.)$ invariant under the action of $S_3$ (see equation \eqref{eq:WWWU}) that allows us to identify $W$ and $W^*$. Thus:
\[
\begin{split}
\frgl(W\oplus E)&\simeq (W\otimes W^*)\oplus(E\otimes E^*)\oplus (W\otimes E^*)\oplus(E\otimes W^*)\\
 &\simeq \frgl(W)\oplus\frgl(E)\oplus (W\otimes E^*)\oplus (E\otimes W^*).
\end{split}
\]
The action of $S_3$ on $\calV$ induces an action of $S_3$ by conjugation on $\frgl(\calV)$, which leaves fixed $\frgl(E)$. Under this action, $\frg=\frgl(W\oplus E)$ decomposes as the direct sum of the eigenspaces for $\varphi$, and the $\omega$ eigenspace is, with the identifications above,
\[
\frg_\omega=\espan{w_{\omega^2}\otimes(w_{\omega^2}\,\vert\, .)}\oplus (w_\omega\otimes E^*)\oplus (E\otimes (w_\omega\,\vert\, .)).
\]
Besides:
\[
\begin{split}
[\tau\bigl(w_{\omega^2}\otimes(w_{\omega^2}\,\vert\, .)\bigr),\tau\bigl(w_\omega\otimes f\bigr)]&=
  [w_{\omega}\otimes(w_{\omega}\,\vert\, .),w_{\omega^2}\otimes f]\\
  &=w_\omega\otimes f,\\[4pt]
[\tau\bigl(w_{\omega^2}\otimes(w_{\omega^2}\,\vert\, .)\bigr),
  \tau\bigl(e\otimes (w_\omega\,\vert\,.)\bigr)]&=
  [w_{\omega}\otimes(w_{\omega}\,\vert\, .),e\otimes (w_{\omega^2}\,\vert\,.)]\\
  &=-e\otimes (w_\omega\,\vert\,.)\\[4pt]
[\tau(\bigl(e\otimes (w_\omega\,\vert\,.)\bigr),\tau\bigl(w_\omega\otimes f\bigr)]
 &=[e\otimes (w_{\omega^2}\,\vert\, .),w_{\omega^2}\otimes f]\\
 &=-f(e)w_{\omega^2}\otimes(w_{\omega^2}\,\vert\,.).
\end{split}
\]
This means that the attached generalized Malcev algebra can be identified, as a vector space, to $ka\oplus E\oplus E^*$, by means of $a\leftrightarrow w_{\omega^2}\otimes(w_{\omega^2}\,\vert\,.)$, $e\leftrightarrow e\otimes(w_\omega\,\vert\, .)$ and $f\leftrightarrow w_\omega\otimes f$, for any $e\in E$ and $f\in E^*$. Then the anticommutative bilinear product becomes:
\[
ae=-e,\quad af=f,\quad fe=f(e)a
\]
for any $e\in E$ and $f\in E^*$. If $\dim E\geq 2$, this is not a Malcev algebra.

Also,
\[
\begin{split}
[w_{\omega^2}\otimes(w_{\omega^2}\,\vert\, .),
  \tau(w_{\omega^2}\otimes(w_{\omega^2}\,\vert\, .)]&=
  [w_{\omega^2}\otimes(w_{\omega^2}\,\vert\, .),
    w_{\omega}\otimes(w_{\omega}\,\vert\, .)]\\
  &=w_{\omega^2}\otimes(w_{\omega}\,\vert\, .)
      -w_{\omega}\otimes(w_{\omega^2}\,\vert\, .),\\[4pt]
[w_{\omega^2}\otimes(w_{\omega^2}\,\vert\, .),\tau\bigl(w_\omega\otimes f\bigr)]
  &=0,\\
[w_{\omega^2}\otimes(w_{\omega^2}\,\vert\, .),
 \tau\bigl(e\otimes (w_\omega\,\vert\,.)\bigr)]&=0,\\
[w_\omega\otimes f_1,w_\omega\otimes f_2]&=0,\\
[e_1\otimes (w_\omega\,\vert\,.),
    \tau\bigl(e_2\otimes (w_\omega\,\vert\,.)\bigr)]&=0,\\[4pt]
[w_\omega\otimes f,\tau\bigl(e\otimes (w_\omega\,\vert\,.)\bigr)]
  &=[w_\omega\otimes f,e\otimes(w_{\omega^2}\,\vert\,.)]\\
  &=f(e)w_\omega\otimes(w_{\omega^2}\,\vert\, .)-e\otimes f,\\[4pt]
[e\otimes(w_\omega\,\vert\,.),\tau(w_\omega\otimes f)]
  &=[e\otimes(w_\omega\,\vert\,.),w_{\omega^2}\,\vert\, .)]\\
  &=e\otimes f-f(e)w_{\omega^2}\otimes(w_\omega\,\vert\, .),
\end{split}
\]
for any $e,e_1,e_1\in E$ and $f,f_1,f_2\in E^*$. With our identification of $\frg_\omega$ with $ka\oplus E\oplus E^*$, it follows easily that the nonzero triple products here are given by:
\[
\begin{split}
&\{a,a,a\}=2a,\ \{a,a,e\}=-e,\ \{a,a,f\}=-f,\\
&\{e,f,a\}=-f(e)a=\{f,e,a\},\\
&\{e,f,e'\}=f(e')e+f(e)e',\quad\{f,e,e'\}=-f(e')e,\\
&\{e,f,f'\}=-f'(e)f,\quad\{f,e,f'\}=f'(e)f+f(e)f',
\end{split}
\]
for any $e,e'\in E$ and $f,f'\in E^*$.
\hfill\qed
\end{example}

\medskip

Before proceeding with more examples, let us show a final connection
of the generalized Malcev algebras to other kind of algebraic
structures.

\begin{remark}\label{re:LieYamaguti}
Let $(M,xy,\{x,y,z\})$ be a generalized Malcev algebra and
let
\[
\frg(M)=(U\otimes \frd^+)\oplus(U'\otimes \frd^-)\oplus (W\otimes M)
\]
be the associated Lie algebra with $S_3$-symmetry. Then $\tau$
becomes an order two automorphism which grades $\frg(M)$ over
$\bZ_2$. The fixed subalgebra by $\tau$ is
\[
\frg_0(M)=(U\otimes \frd^+)\oplus(w_+\otimes M),
\]
whose Lie bracket works as follows:
\[
\begin{split}
[u\otimes d_1,u\otimes d_2]&=u\otimes [d_1,d_2],\\
[u\otimes d_1,w_+\otimes x]&=w_+\otimes d(x),\\
[w_+\otimes x,w_+\otimes y]&=u\otimes [x,y,.]+w_+\otimes xy,
\end{split}
\]
for any $d,d_1,d_2\in \frd^+$ and $x,y\in M$, where
$[x,y,.]=\{x,y,.\}-\{y,x,.\}$. This algebra is isomorphic to the Lie
algebra
\[
\tilde\frg_0(M)=[M,M,.]\oplus M
\]
with $[M,M,.]=\espan{[x,y,.]: x,y\in M}\subseteq\frgl(M)$, and bracket
given by imposing that $[M,M,.]$ is a subalgebra, $M$ its natural
module, and $[x,y]=[x,y,.]+xy$ for any $x,y\in M$. Thus
$(M,xy,[x,y,z])$ is a Lie-Yamaguti algebra (or general Lie triple
system in Yamaguti's notation \cite{Yamaguti}) and $\tilde\frg_0(M)$ is its
standard enveloping Lie algebra (see \cite{Kinyon}).\hfill\qed
\end{remark}

\bigskip

Another family of examples of generalized Malcev algebras comes from
the second line of Freudenthal's Magic Square.

Let $K=ke+kz$ be the quadratic commutative associative algebra with
$ex=x$ for any $x$ and $z^2=-3e$. If $\sqrt{-3}\not\in k$, then $K$
is the field extension $k[\sqrt{-3}]$, otherwise $K$ is isomorphic
to $k\times k$. This is a composition algebra relative to the norm
$q$, that is $q(xy)=q(x)q(y)$ for any $x,y\in K$, where $q(e)=1$, $q(z)=3$ and
$q(e,z)\bigl(=q(e+z)-q(e)-q(z)\bigr)=0$. Let $(K,\bullet,q)$ be the
associated para-Hurwitz algebra, where $x\bullet y=\overline{xy}$
($\bar e=e$, $\bar z=-z$), that is,
\[
e\bullet e=e,\quad e\bullet z=-z=z\bullet e,\quad z\bullet z=-3e.
\]
The set of nonzero idempotents of $(K,\bullet)$ is
$\{e,-\frac{1}{2}e+\frac{1}{2}z,-\frac{1}{2}e-\frac{1}{2}z\}$, and
its group of automorphisms $\Aut(K,\bullet)$ permutes these
idempotents. From here it follows that $\Aut(K,\bullet)\simeq S_3$,
where the action of $\varphi$ and $\tau$ are given by:
\[
\left\{\begin{aligned} &\tau(e)=e,\ \tau(z)=-z,\\
 &\varphi(e)=-\frac{1}{2}e+\frac{1}{2}z,\
 \varphi(z)=-\frac{3}{2}e-\frac{1}{2}z.\end{aligned}\right.
\]
Actually, as a module for $S_3$, $K$ is isomorphic to the two
dimensional module $W$: $e\leftrightarrow w_+$, $z\leftrightarrow
w_-$. Under this linear isomorphism, the multiplication in
\eqref{eq:wsbullet} corresponds to the multiplication $\bullet$ on
$K$, while the symmetric bilinear form $(.\,\vert\, .)$ in
\eqref{eq:wsbilinear} corresponds to the polar form $q(.,.)$ of the
norm $q$ of $K$.

The triality Lie algebra of $(K,\bullet,q)$ is
\[
\begin{split}
\tri(&K,\bullet,q)\\
&=\{(d_0,d_1,d_2)\in\frso(K,q): d_0(x\bullet y)=d_1(x)\bullet
y+x\bullet d_2(y)\ \forall x,y\in K\}.
\end{split}
\]
But the orthogonal Lie algebra $\frso(K,q)$ has dimension $1$, being
spanned by $\sigma=\sigma_{e,z}=q(e,.)z-q(z,.)e$. That is,
\[
\sigma:\begin{cases} e\mapsto q(e,e)z-q(e,z)e=2z,\\
 z\mapsto q(e,z)z-q(z,z)e=-6e. \end{cases}
\]
It follows that
\[
\tri(K,\bullet,q)=\{(\alpha_0\sigma,\alpha_1\sigma,\alpha_2\sigma):
\alpha_0,\alpha_1,\alpha_2\in k,\ \alpha_0+\alpha_1+\alpha_2=0\}.
\]
Under the identification above of $W$ and $K$, the action of
$\sigma$ on $K$ corresponds exactly to the action of $u'$ on $W$
($u'\diamond w_+=2w_+$, $u'\diamond w_-=-6w_-$ because of
\eqref{eq:UprimeUprime}).

The action of $S_3$ on $K$ induces an action on $\tri(K,\bullet,q)$
by
\[
\gamma\bigl((d_0,d_1,d_2)\bigr)=(\gamma d_0\gamma^{-1},\gamma
d_1\gamma^{-1},\gamma d_2\gamma^{-1}),
\]
for any $\gamma\in S_3=\Aut(K,\bullet)$ and
$(d_0,d_1,d_2)\in\tri(K,\bullet,q)$. But $\tau\sigma\tau=-\sigma$
and $\varphi\sigma\varphi^2=\sigma$, so $\tri(K,\bullet,q)$
decomposes into the direct sum of two copies of the alternating
module. Thus, as $S_3$-modules, consider the identification:
\begin{equation}\label{eq:triKUprime}
\begin{split}
\tri(K,\bullet,q)&\simeq U'\otimes (k^3)_0\\
(\alpha_0\sigma,\alpha_1\sigma,\alpha_2\sigma)&\leftrightarrow
u'\otimes (\alpha_0,\alpha_1,\alpha_2),
\end{split}
\end{equation}
where $(k^3)_0=\{(\alpha_0,\alpha_1,\alpha_2)\in k^3:
\alpha_0+\alpha_1+\alpha_2=0\}$.

Given any other symmetric composition algebra $(S,*,\hat q)$ (see
\cite[Chapter VIII]{KMRT} and the references there in), a Lie
algebra is constructed in \cite{EldIbero04} on the vector space
\[
\frg(K,S)=\tri(K,\bullet,q)\oplus\tri(S,*,\hat
q)\oplus\Bigl(\oplus_{i=0}^2\iota_i(K\otimes S)\Bigr),
\]
where $\iota_i(K\otimes S)$ denotes a copy of $K\otimes S$,
$i=0,1,2$. The Lie bracket is defined by:
\begin{itemize}
\item $\tri(K,\bullet,,q)$ and $\tri(S,*,\hat q)$ are two commuting Lie
subalgebras of $\frg(K,S)$ (so
 $[\tri(K,\bullet,,q),\tri(S,*,\hat q)]=0$),

\item $[(d_0,d_1,d_2),\iota_i(a\otimes x)]=\iota_i(d_i(a)\otimes x)$,
$[(\hat d_0,\hat d_1,\hat d_2),\iota_i(a\otimes x)]=\iota_i(a\otimes
\hat d_i(x))$, for $i=0,1,2$, $a\in K$, $x\in S$,
$(d_0,d_1,d_2)\in\tri(K,\bullet,q)$ and $(\hat d_0,\hat d_1,\hat
d_2)\in\tri(S,*,\hat q)$,

\item $[\iota_i(a\otimes x),\iota_{i+1}(b\otimes
y)]=\iota_{i+2}((a\bullet b)\otimes(x*y))$, for $i=0,1,2$ (indices
modulo $3$), $a,b\in K$ and $x,y\in S$,

\item $[\iota_i(a\otimes x),\iota_i(b\otimes y)]=\hat
q(x,y)\theta^i(t_{a,b})+q(a,b)\theta^i(t_{x,y})$ for $i=0,1,2$,
$a,b\in K$ and $x,y\in S$, where $\theta:(d_0,d_1,d_2)\mapsto
(d_2,d_0,d_1)$,
$t_{a,b}=\bigl(\sigma_{a,b},\frac{1}{2}q(a,b)I-R_aL_b,\frac{1}{2}q(a,b)I-L_aR_b\bigr)$ and similarly for $t_{x,y}$.
As usual $R$ and $L$ denote right and left multiplications, $I$ is
the identity mapping, and $\sigma_{a,b}=q(a,.)b-q(b,.)a$).
\end{itemize}

Depending on the dimension of $S$ being $1$, $2$, $4$ or $8$, the
Lie algebra $\frg(K,S)$ is a form respectively of $\frsl_3(k)$,
$\frsl_3(k)\oplus\frsl_3(k)$, $\frsl_6(k)$ or the simple Lie algebra
of type $E_6$. That is, we get the second line of Freudenthal's
Magic Square \cite{EldIbero04}.

Now, the action of $S_3$ by automorphisms of $K$ extends to an
action of $S_3$ by automorphisms of $\frg(K,S)$, with the action of any $\gamma\in S_3$ is given by:
\begin{itemize}
\item $\gamma\bigl(\iota_i(a\otimes
x)\bigr)=\iota_i(\gamma(a)\otimes x)$, for any $a\in K$, $x\in S$,
$i=0,1,2$,

\item $\gamma\bigl((d_0,d_1,d_2)\bigr)=(\gamma d_0\gamma^{-1},\gamma
d_1\gamma^{-1},\gamma d_2\gamma^{-1})$ for any $(d_0,d_1,d_2)$ in
\newline $\tri(K,\bullet,q)$,

\item $\gamma\bigl((\hat d_0,\hat d_1,\hat d_2)\bigr)
=(\hat d_0,\hat d_1,\hat d_2)$ for any $(\hat d_0,\hat d_1,\hat
d_2)\in\tri(S,*,\hat q)$.
\end{itemize}

Thus, $\tri(S,*,q)$ is a trivial module for $S_3$,
$\tri(K,\bullet,q)$ is already known to be isomorphic to two copies
of the alternating module by \eqref{eq:triKUprime}:
$\tri(K,\bullet,q)\simeq U'\otimes (k^3)_0$, while
$\oplus_{i=0}^2\iota_i(K\otimes S)$ is a sum of $3\times \dim S$
copies of the two dimensional irreducible module $W\simeq K$.

Consider then the natural identification:
\[
\begin{split}
\frg(K,S)&\simeq \bigl(U\otimes\tri(S,*,q)\bigr)
   \oplus\bigl(U'\otimes (k^3)_0\bigr))\oplus
   \bigl(W\otimes (\oplus_{i=0}^2\iota_i(S))\bigr),\\
 (\hat d_0,\hat d_1,\hat d_2)&\leftrightarrow u\otimes (\hat d_0,\hat d_1,\hat
 d_2),\\
 (\alpha_0\sigma,\alpha_1\sigma,\alpha_2\sigma)&\leftrightarrow
 u'\otimes (\alpha_0,\alpha_1,\alpha_2),\\
 \iota_i(a\otimes x)&\leftrightarrow a\otimes \iota_i(x),
\end{split}
\]
for $(\hat d_0,\hat d_1,\hat d_2)\in\tri(S,*,\hat q)$,
$(\alpha_0,\alpha_1,\alpha_2)\in(k^3)_0$, $\sigma=\sigma_{e,z}$ as
above, $i=0,1,2$, $a\in K\simeq W$, and $x\in S$.

The Lie bracket on $\frg(K,S)$ translates into a bracket on the
right which proves the validity of the following example:

\begin{example} \textbf{(Second line of Freudenthal's Magic Square)}\newline
Let $(S,*,\hat q)$ be a symmetric composition algebra and let
$M=\oplus_{i=0}^2\iota_i(S)$, be the direct sum of three copies of
$S$. Then $\bigl(M,XY,\{X,Y,Z\}\bigr)$ is a generalized Malcev
algebra, where for $i,j,k\in\{0,1,2\}$ and $x,y,z\in S$:
\begin{itemize}
\item $\iota_i(x)\iota_{i+1}(y)=\iota_{i+2}(x*y)$,

\item $\iota_i(x)\iota_i(y)=0$,

\item $\{\iota_i(x),\iota_j(y),\iota_k(z)\}=0$, if $i\ne j$,

\item
$\{\iota_i(x),\iota_i(y),\iota_{i+1}(z)\}=-\iota_{i+1}((y*z)*x)$,

\item
$\{\iota_i(x),\iota_i(y),\iota_{i+2}(z)\}=-\iota_{i+2}(x*(z*y))$,

\item
$\{\iota_i(x),\iota_i(y),\iota_i(z)\}=\iota_i\bigl(\hat q(x,z)y-\hat
q(y,z)x+\hat q(x,y)z\bigr)$.\hfill\qed
\end{itemize}
\end{example}

\smallskip

Notice that, for example, for $x,y\in S$,
\[
[\iota_i(e\otimes
x),\iota_i(e\otimes y)]=q(e,e)\theta^i(t_{x,y})=2\theta^i(t_{x,y}),
\]
and this translates to
\[
[w_+\otimes x,w_+\otimes y]=2u\otimes
\theta^i(t_{x,y})=(w_+\,\vert\, w_+)u\otimes \theta^i(t_{x,y}).
\]
This
gives $\iota_i(x)\iota_i(y)=0$ and
\[
d_{\iota_i(x),\iota_i(y)}^+(\iota_j(z))=
 \begin{cases}
  \iota_i\bigl(\sigma_{x,y}(z)\bigr)&\text{if $j=i$,}\\
  \iota_{i+1}\bigr(\frac{1}{2}\hat q(x,y)z-(y*z)*x\bigr)&\text{if $j=i+1$ (modulo $3$),}\\
  \iota_{i+2}\bigr(\frac{1}{2}\hat q(x,y)z-x*(z*y)\bigr)&\text{if $j=i+2$
  (modulo $3$).}\end{cases}
\]
Also $[\iota_i(e\otimes x),\iota_i(z\otimes y)]=\hat
q(x,y)\theta^i(t_{e,z})$, and since
$t_{e,z}=(\sigma,-\frac{1}{2}\sigma,-\frac{1}{2}\sigma)$, this gives
$[w_+\otimes x,w_-\otimes y]=\hat q(x,y) u'\otimes
\theta^i\bigl(1,-\frac{1}{2},-\frac{1}{2}\bigr)$ (where again
$\theta$ denotes the cyclic permutation of the components). Thus,
\[
d_{\iota_i(x),\iota_i(y)}^-(\iota_j(z))=
 \begin{cases}
  \iota_i\bigl(\hat q(x,y)z\bigr)&\text{if $j=i$,}\\
  \iota_{i+1}\bigl(-\frac{1}{2}\hat q(x,y)z\bigr)&\text{if $j=i+1$ (modulo $3$),}\\
  \iota_{i+2}\bigl(-\frac{1}{2}\hat q(x,y)z\bigr)&\text{if $j=i+2$
  (modulo $3$).}
 \end{cases}
\]
The different items in the example above follow at once from these
computations.


%

\def\cprime{$'$}
\providecommand{\bysame}{\leavevmode\hbox
to3em{\hrulefill}\thinspace}
\providecommand{\MR}{\relax\ifhmode\unskip\space\fi MR }
\providecommand{\MRhref}[2]{%
  \href{http://www.ams.org/mathscinet-getitem?mr=#1}{#2}
} \providecommand{\href}[2]{#2}

\end{document}